\documentclass[11pt]{amsart}
\usepackage[english]{babel}
\usepackage[T1]{fontenc}
\usepackage[latin1]{inputenc}
\usepackage{graphicx}
\usepackage{amsmath,amssymb,amsthm,mathrsfs,txfonts,ifthen,xr}
\usepackage{soul}
\usepackage{array, multirow}
\usepackage{stmaryrd}
\usepackage{color}
\usepackage{hyperref}
\usepackage{enumerate}
\usepackage[all]{xy}
\usepackage{tikz-cd}
\usetikzlibrary{decorations.pathreplacing}

\usepackage{ytableau}

\theoremstyle{plain}
\newtheorem{theo}{Theorem}[section]
\newtheorem{lemma}[theo]{Lemma}
\newtheorem{proposition}[theo]{Proposition}
\newtheorem{corollary}[theo]{Corollary}

\theoremstyle{definition}
\newtheorem{definition}[theo]{Definition}

\newtheorem{notation}[theo]{Notation} 

\newtheorem{example}[theo]{Example}

\theoremstyle{remark}
\newtheorem{remark}[theo]{Remark}

\def\A{{\rm A}}
\def\B{{\rm B}}
\def\C{{\rm C}}
\def\D{{\rm D}}
\def\d{{\rm d}}

\def\G{{\rm G}}
\def\H{{\rm H}}

\def\K{{\rm K}}
\def\L{{\rm L}}
\def\M{{\rm M}}

\def\O{{\rm O}}
\def\P{{\rm P}}

\def\S{{\rm S}}

\def\T{{\rm T}}
\def\U{{\rm U}}
\def\V{{\rm V}}
\def\W{{\rm W}}
\def\X{{\rm X}}
\def\Y{{\rm Y}}
\def\Z{{\rm Z}}

\def\tr{{\rm tr}}

\def\Sp{{\rm Sp}}
\def\GL{{\rm GL}}	
\def\Id{{\rm Id}}

\def\Im{{\rm Im}}

\def\Sym{{\rm Sym}}

\def\Mat{{\rm Mat}}
\def\Bar{{\rm Bar}}

\def\trd{{\rm trd}}
\def\log{{\rm log}}

\def\Spec{{\rm Spec}}
\def\diag{{\rm diag}}

\def\Span{{\rm Span}}

\allowdisplaybreaks
\title{Kubo-Ando Means of Finite Order over Real Division Algebras}
\author{Jose Franco}
\address{Department of Mathematics and Statistics\\ University of North Florida \\ 1 UNF Drive \\ Jacksonville \\ FL 32224 \\ USA}
\email{jose.franco@unf.edu}
\author{Allan Merino}
\address{Department of Mathematics and Statistics\\ University of North Florida \\ 1 UNF Drive \\ Jacksonville \\ FL 32224 \\ USA}
\email{allan.merino@unf.edu}



\date{}

\begin{document}

\begin{abstract}

In this paper, we study Kubo-Ando means of finite order over the real division algebras $\mathbb{R}$, $\mathbb{C}$, and $\mathbb{H}$. We establish an explicit bijection between Kubo-Ando means of order $n$ and normalized positive functions $f$ such that both $f$ and its transpose $f^{\top}(t)=tf(t^{-1})$ are matrix monotone of order $n$. We also prove that matrix monotonicity of a fixed order is independent of the underlying division algebra. We then investigate finite-order Kubo-Ando means under the canonical embeddings relating quaternionic, complex, and real positive definite matrices. These embeddings preserve functional calculus and the Loewner order, inducing natural correspondences between means on the original and embedded cones. We characterize their extensions to the ambient cones and study the geometry of the embeddings with respect to the Log-Euclidean metric.

\noindent As an application, we obtain explicit affine formulas for Kubo-Ando means of order $2$ on $\mathscr{P}_{2}(\mathbb{D})$ and transfer them to the associated embedded cones of $4\times4$ positive definite matrices. In particular, we derive trace-determinant formulas for the geometric mean and show that they do not, in general, extend to the full ambient cones.

\end{abstract}

\maketitle

\tableofcontents

\section{Introduction}

The theory of positive definite matrices plays a central role in matrix analysis, operator theory, differential geometry, and mathematical physics. Among the structures defined on the cone of positive definite matrices, operator means introduced by Kubo and Ando \cite{KUBOANDO} occupy a distinguished position. Their theory provides a framework for binary operations compatible with the Loewner order and congruence transformations, including the classical arithmetic, harmonic, and geometric means. A fundamental result establishes a one-to-one correspondence between operator means and operator monotone functions on $\left(0\,,+\infty\right)$. In finite dimensions, requiring monotonicity only for matrices of a fixed size leads naturally to matrix monotone functions and Kubo-Ando means of finite order.

\noindent While most of the classical literature focuses on real and complex matrices, many geometric and spectral properties extend naturally to the quaternionic setting. In particular, the cone $\mathscr{P}_{n}(\mathbb{H})$ of positive quaternionic Hermitian matrices admits spectral decomposition, functional calculus, and geometric structures associated with matrix means. Nevertheless, quaternionic matrix means have received comparatively little attention.

\noindent The main purpose of this paper is to study Kubo-Ando means of finite order over the three finite-dimensional real division algebras $\mathbb{R}$, $\mathbb{C}$, and $\mathbb{H}$. Unlike the classical theory, we consider means on positive definite matrices of a fixed size, whose representing functions need not be operator monotone. We first establish a correspondence theorem showing that Kubo-Ando means of order $n$ are uniquely determined by normalized functions $f$, with $f(1)=1$, such that both $f$ and its transpose $f^{\top}(t)=tf(t^{-1})$ are matrix monotone of order $n$. Conversely, every such function determines a unique mean of order $n$. We also prove that matrix monotonicity of a fixed order is independent of the underlying division algebra: a function is matrix monotone of order $n$ over $\mathbb{R}$, $\mathbb{C}$, or $\mathbb{H}$ if and only if it is matrix monotone of the same order over the other two.

\noindent We then introduce a second notion of Kubo-Ando means defined on distinguished subsets of positive definite matrices, with congruence invariance restricted to an appropriate subgroup of the general linear group. This framework allows us to study the preservation of means under the classical embeddings
\begin{equation*}
\Psi_{1}: \Mat_{n}(\mathbb{C}) \hookrightarrow \Mat_{2n}(\mathbb{R})\,, \qquad \Psi_{2}: \Mat_{n}(\mathbb{H}) \hookrightarrow \Mat_{2n}(\mathbb{C})\,,
\end{equation*}
without requiring their representing functions to be matrix monotone on the entire ambient cone. These embeddings preserve algebraic operations, adjoints, positivity, and functional calculus, and identify $\mathscr{P}_{n}(\mathbb{C})$ and $\mathscr{P}_{n}(\mathbb{H})$ with subcones of $\mathscr{P}_{2n}(\mathbb{R})$ and $\mathscr{P}_{2n}(\mathbb{C})$, respectively, characterized by compatibility with natural complex and quaternionic structures\,.

\noindent Using the compatibility of these embeddings with the Loewner order and functional calculus, we show that every Kubo-Ando mean of order $n$ on $\mathscr{P}_{n}(\mathbb{C})$ induces a corresponding mean on the embedded cone $\widetilde{\mathscr{P}}_{n}(\mathbb{C}) \subseteq \mathscr{P}_{2n}(\mathbb{R})$. Similarly, every Kubo-Ando mean of order $n$ on $\mathscr{P}_{n}(\mathbb{H})$ induces a corresponding mean on $\widetilde{\mathscr{P}}_{n}(\mathbb{H})\subset\mathscr{P}_{2n}(\mathbb{C})$. We characterize when these induced means extend to the corresponding ambient cones in terms of matrix monotonicity of higher order. We also study several geometric structures preserved by the embeddings, including the Log-Euclidean metric and Log-Euclidean barycenters\,.

\medskip

\noindent As an application of the preceding results, we investigate explicit formulas for Kubo-Ando means in low dimensions. In the case of $\mathscr{P}_{2}(\mathbb{D})$, the functional calculus implies that every Kubo-Ando mean of order $2$ admits an explicit affine expression in terms of the two matrices involved. More precisely, if $\sigma$ is a Kubo-Ando mean of order $2$ with representing function $f$, then
\begin{equation*}
\A \sigma \B = \alpha_{f}(\X)\A+\beta_{f}(\X)\B\,, \qquad \X=\A^{-\frac{1}{2}}\B\A^{-\frac{1}{2}}\,,
\end{equation*}
where the coefficients $\alpha_{f}(\X)$ and $\beta_{f}(\X)$ depend only on the eigenvalues of $\X$ and the function $f$, and hence only on the trace and determinant of $\X$. An explicit formula is given in Proposition~\ref{DecompositionMeanAlphaBeta}\,.

\noindent Using the canonical embeddings relating quaternionic, complex, and real positive definite matrices, we transfer these formulas to the corresponding embedded cones of $4\times4$ positive definite matrices. In particular, the induced means on the embedded cones $\Psi_{1}(\mathscr{P}_{2}(\mathbb{C}))$ and $\Psi_{2}(\mathscr{P}_{2}(\mathbb{H}))$ admit explicit affine expressions of the form
\begin{equation*}
\X \widetilde{\sigma} \Y = \alpha_{f}(\T)\X+\beta_{f}(\T)\Y\,, \qquad \T=\X^{-\frac{1}{2}}\Y\X^{-\frac{1}{2}}\,,
\end{equation*}
where the coefficients are determined explicitly by the spectrum of $\T$. In the geometric mean case, this yields trace-determinant formulas for real $4\times4$ matrices in the image of $\mathscr{P}_{2}(\mathbb{C})$, recovering a result of \cite{CHOIKIMLIM}. We further show that such affine formulas are specific to the embedded subcone $\Psi_{1}(\mathscr{P}_{2}(\mathbb{C}))$ and do not extend, in general, to arbitrary elements of $\mathscr{P}_{4}(\mathbb{R})$. Finally, using the reduced trace and the Moore determinant, we obtain quaternionic analogues of these formulas for matrices in $\mathscr{P}_{2}(\mathbb{H})$.

\medskip

\noindent The paper is organized as follows. In Section~2, we recall basic facts on positive definite matrices over $\mathbb{R}$, $\mathbb{C}$, and $\mathbb{H}$, together with the embeddings $\Psi_{1}$ and $\Psi_{2}$, and we show that these embeddings are compatible with the functional calculus and are order embeddings for the Loewner order. In Section~3, we prove that the map
\begin{equation*}
\Phi: \mathcal{M}_{n}(\mathbb{D}) \longrightarrow \mathcal{F}_{n}(\mathbb{D})\,, \qquad \sigma \longmapsto f\,, \quad \text{where } \quad \Id_{n}\sigma\left(t\Id_{n}\right) = f(t)\Id_{n}\,,
\end{equation*}
is a bijection between the set $\mathcal{M}_{n}(\mathbb{D})$ of Kubo-Ando means of order $n$ on $\mathscr{P}_{n}(\mathbb{D})$ and the class $\mathcal{F}_{n}(\mathbb{D})$ of functions $f: \left(0, \infty\right) \to \left(0, \infty\right)$ with $f(1) = 1$ such that $f$ and $f^{\top}$ are matrix monotone of order $n$, the inverse being given by $f \longmapsto \sigma_{f}$, with
\begin{equation*}
\A \sigma_{f} \B = \A^{\frac{1}{2}}f\left(\A^{-\frac{1}{2}}\B\A^{-\frac{1}{2}}\right)\A^{\frac{1}{2}}\,, \qquad \left(\A\,, \B \in \mathscr{P}_{n}(\mathbb{D})\right)\,.
\end{equation*}
We show that matrix monotonicity of order $n$ does not depend on $\mathbb{D} \in \left\{\mathbb{R}\,, \mathbb{C}\,, \mathbb{H}\right\}$, so that $\mathcal{F}_{n}(\mathbb{R}) = \mathcal{F}_{n}(\mathbb{C}) = \mathcal{F}_{n}(\mathbb{H})$. We then study the means induced by $\Psi_{1}$ and $\Psi_{2}$ on the embedded cones $\widetilde{\mathscr{P}}_{n}(\mathbb{C})$ and $\widetilde{\mathscr{P}}_{m}(\mathbb{H})$, and characterize when they extend to the ambient cones $\mathscr{P}_{2n}(\mathbb{R})$ and $\mathscr{P}_{2m}(\mathbb{C})$, namely if and only if $f \in \mathcal{F}_{2n}(\mathbb{R})$ (resp. $f \in \mathcal{F}_{2m}(\mathbb{C})$). In Section~4, we study metric properties of the embeddings and show that they preserve the Log-Euclidean geometry. Section~5 is devoted to Kubo-Ando means on $\mathscr{P}_{2}(\mathbb{D})$ and their compatibility with $\Psi_{1}$ and $\Psi_{2}$. Finally, in Section~6, we discuss approximation results related to the embedded complex structures inside $\mathscr{P}_{2n}(\mathbb{R})$\,.

\section{Preliminaries}

Let $\mathbb{D} \in \left\{\mathbb{R}\,, \mathbb{C}\,, \mathbb{H}\right\}$, and let $\iota: \mathbb{D} \to \mathbb{D}$ be the map given by
\begin{equation*}
\iota(\X) = \overline{\X}\,, \qquad \left(\X \in \mathbb{D}\right)\,.
\end{equation*}
For all $n \in \mathbb{N}^{*}$, we denote by $\Mat_{n}(\mathbb{D})$ the set of $n$ by $n$ matrices with entries in $\mathbb{D}$, and let $\GL_{n}(\mathbb{D})$ be the set of invertible matrices in $\Mat_{n}(\mathbb{D})$\,.

\begin{notation}

For all $\A = \left(a_{i, j}\right)_{1 \leq i, j \leq n} \in \Mat_{n}(\mathbb{D})$, we denote by $\A^{*}$ the matrix of $\Mat_{n}(\mathbb{D})$ given by
\begin{equation*}
a^{*}_{i, j} = \iota(a_{j, i})\,, \qquad \left(1 \leq i\,, j \leq n\right)\,.
\end{equation*}

\end{notation}

\begin{definition}

Let $\A \in \Mat_{n}(\mathbb{D})$. We say that
\begin{itemize}
\item $\A$ is Hermitian if $\A = \A^{*}$\,,
\item $\A$ is positive (and write $\A > 0$) if $x^{*}\A x > 0$ for all non-zero $x \in \mathbb{D}^{n}$\,.
\item $\A$ is positive semidefinite (and write $\A \geq 0$) if $x^{*}\A x \geq 0$ for all non-zero $x \in \mathbb{D}^{n}$\,.
\end{itemize}

\end{definition}

\noindent We denote by $\mathfrak{p}_{n}(\mathbb{D})$ the subset of $\Mat_{n}(\mathbb{D})$ consisting of Hermitian matrices, and by $\mathscr{P}^{0}_{n}(\mathbb{D})$ and $\mathscr{P}_{n}(\mathbb{D})$ the subsets of $\mathfrak{p}_{n}(\mathbb{D})$ given by
\begin{equation*}
\mathscr{P}^{0}_{n}(\mathbb{D}) = \left\{\X \in \mathfrak{p}_{n}(\mathbb{D})\,, \X \geq 0\right\}\,, \qquad \mathscr{P}_{n}(\mathbb{D}) = \left\{\X \in \mathfrak{p}_{n}(\mathbb{D})\,, \X > 0\right\}\,.
\end{equation*}

\noindent On the set $\mathscr{P}_{n}(\mathbb{D})$, we denote by $\leq$ the Loewner order, that is a partial order defined by
\begin{equation*}
\A \leq \B \qquad \Leftrightarrow \qquad \B - \A \in \mathscr{P}^{0}_{n}(\mathbb{D})\,.
\end{equation*}

\noindent We denote by $\exp_{\mathbb{D}}: \Mat_{n}(\mathbb{D}) \to \GL_{n}(\mathbb{D})$ the map given by
\begin{equation*}
\exp_{\mathbb{D}}(\X) = \sum\limits_{k = 0}^{\infty} \frac{\X^{k}}{k!}\,, \qquad \left(\X \in \Mat_{n}(\mathbb{D})\right)\,.
\end{equation*}

\noindent It is well-known (see \cite{BHATIA} for $\mathbb{D} \in \left\{\mathbb{R}\,, \mathbb{C}\right\}$ and \cite[Appendix~A]{FRANCOMERINO} for $\mathbb{D} = \mathbb{H}$) that for all $\X \in \mathfrak{p}_{n}(\mathbb{D})$, we have $\exp_{\mathbb{D}}(\X) \in \mathscr{P}_{n}(\mathbb{D})$ and the corresponding map
\begin{equation*}
\exp_{\mathbb{D}}: \mathfrak{p}_{n}(\mathbb{D}) \to \mathscr{P}_{n}(\mathbb{D})
\end{equation*}
is bijective. We denote by $\log_{\mathbb{D}}$ the inverse of $\exp_{\mathbb{D}}$, i.e. 
\begin{equation*}
\exp_{\mathbb{D}}\left(\log_{\mathbb{D}}(\X)\right) = \X\,, \qquad \log_{\mathbb{D}}\left(\exp_{\mathbb{D}}(\Y)\right) = \Y\,, \qquad \left(\X \in \mathscr{P}_{n}(\mathbb{D})\,, \Y \in \mathfrak{p}_{n}(\mathbb{D})\right)\,.
\end{equation*}

\begin{definition}

For all $\X \in \mathscr{P}_{n}(\mathbb{D})$ and $t \in \mathbb{R}$, we denote by $\X^{t}_{\mathbb{D}}$ the matrix in $\mathscr{P}_{n}(\mathbb{D})$ given by
\begin{equation*}
\X^{t}_{\mathbb{D}} := \exp_{\mathbb{D}}\left(t\log_{\mathbb{D}}(\X)\right)\,.
\end{equation*}

\end{definition}

\noindent Let $\Psi_{1}: \mathbb{C} \to \Mat_{2}(\mathbb{R})$ be the map given by
\begin{equation*}
\Psi_{1}(a + ib) = \begin{pmatrix} a & b \\ -b & a \end{pmatrix}\,, \qquad \left(a\,, b \in \mathbb{R}\right)\,.
\end{equation*}
The map $\Psi_{1}$ is a monomorphism of algebras such that
\begin{equation*}
\Psi_{1}(\overline{z}) = \Psi_{1}(z)^{t}\,, \qquad \left(z \in \mathbb{C}\right)\,.
\end{equation*}
Using that $\Mat_{n}(\mathbb{C}) = \Mat_{n}(\mathbb{R}) \oplus i\Mat_{n}(\mathbb{R})$, the map $\Psi_{1}$ can be extended to an injective morphism
\begin{equation*}
\Psi_{1}: \Mat_{n}(\mathbb{C}) \to \Mat_{2n}(\mathbb{R})
\end{equation*}
defined by
\begin{equation*}
\Psi_{1}(\A + i\B) = \begin{pmatrix} \A & \B \\ -\B & \A \end{pmatrix}\,, \qquad \left(\A\,, \B \in \Mat_{n}(\mathbb{R})\right)\,.
\end{equation*}
and such that
\begin{equation}
\Psi_{1}(\X^{*}) = \Psi_{1}(\X)^{t}\,, \qquad \left(\X \in \Mat_{n}(\mathbb{C})\right)\,.
\label{AdjointPsiOne}
\end{equation}

\begin{notation}

We denote by $\K_{2n}$ the matrix in $\Mat_{2n}(\mathbb{D})$ given by
\begin{equation*}
\K_{2n} = \begin{pmatrix} 0 & \Id_{n} \\ -\Id_{n} & 0 \end{pmatrix}\,.
\end{equation*}

\end{notation}

\begin{proposition}

We have $\Psi_{1}\left(\mathscr{P}_{n}(\mathbb{C})\right) \subseteq \mathscr{P}_{2n}(\mathbb{R})$. More precisely, we have
\begin{equation*}
\Psi_{1}\left(\mathscr{P}_{n}(\mathbb{C})\right) = \left\{\X \in \mathscr{P}_{2n}(\mathbb{R})\,, \K_{2n}\X = \X\K_{2n}\right\}\,.
\end{equation*}
\label{PropositionCtoR}

\end{proposition}

\begin{proof}

Let $\X \in \mathscr{P}_{n}(\mathbb{C})$. Using that $\X = \X^{*}$, we get from Equation \eqref{AdjointPsiOne} that $\Psi_{1}(\X)^{t} = \Psi_{1}(\X^{*}) = \Psi_{1}(\X)$, i.e. $\Psi_{1}(\X) \in \mathfrak{p}_{2n}(\mathbb{R})$. It remains to show that $\Psi_{1}(\X)$ is positive. Let $u\,, v \in \mathbb{R}^{n}$, not both zero, and set $z=u+iv \in \mathbb{C}^{n}$. Since $\X$ is Hermitian, we have $\A^{t} = \A$ and $\B^{t} = -\B$. Then
\begin{eqnarray*}
\begin{pmatrix} u \\ -v \end{pmatrix}^{t} \Psi_{1}(\X) \begin{pmatrix} u \\ -v \end{pmatrix} & = & \begin{pmatrix} u^{t} & -v^{t} \end{pmatrix} \begin{pmatrix} \A & \B \\ -\B & \A \end{pmatrix} \begin{pmatrix} u \\ -v \end{pmatrix} = \begin{pmatrix} u^{t} & -v^{t} \end{pmatrix} \begin{pmatrix} \A u-\B v \\ -\B u-\A v \end{pmatrix} \\
& = & u^{t}\A u-u^{t}\B v+v^{t}\B u+v^{t}\A v = u^{t}\A u-u^{t}\B v-u^{t}\B v+v^{t}\A v \\
& = & u^{t}\A u+v^{t}\A v-2u^{t}\B v\,.
\end{eqnarray*}
On the other hand,
\begin{eqnarray*}
z^{*}\X z & = & \left(u^{t}-iv^{t}\right)\left(\A+i\B\right)\left(u+iv\right) = \left(u^{t}-iv^{t}\right)\left(\A u+i\A v+i\B u-\B v\right) \\
& =& u^{t}\A u+iu^{t}\A v+iu^{t}\B u-u^{t}\B v - iv^{t}\A u+v^{t}\A v+v^{t}\B u+iv^{t}\B v\,.
\end{eqnarray*}
Since $\A^{t}=\A$ and $\B^{t}=-\B$, we have
\begin{equation*}
u^{t}\A v = v^{t}\A u\,, \qquad u^{t}\B u=v^{t}\B v = 0\,, \qquad v^{t}\B u = -u^{t}\B v\,.
\end{equation*}
Therefore
\begin{equation*}
z^{*}\X z = u^{t}\A u+v^{t}\A v-2u^{t}\B v\,.
\end{equation*}
Hence
\begin{equation*}
\begin{pmatrix} u \\ -v \end{pmatrix}^{t} \Psi_{1}(\X) \begin{pmatrix} u \\ -v \end{pmatrix} = z^{*}\X z > 0\,,
\end{equation*}
i.e. $\Psi_{1}(\X) \in \mathscr{P}_{2n}(\mathbb{R})$. Using that 
\begin{equation*}
\Im(\Psi_{1}) \subseteq \left\{\begin{pmatrix} \A & \B \\ -\B & \A \end{pmatrix}\,, \A\,, \B \in \Mat_{n}(\mathbb{R})\right\} = \left\{\X \in \Mat_{2n}(\mathbb{R})\,, \X\K_{2n} = \K_{2n}\X\right\}\,,
\end{equation*}
it follows from the previous computations that 
\begin{equation*}
\Psi_{1}\left(\mathscr{P}_{n}(\mathbb{C})\right) \subseteq \left\{\X \in \mathscr{P}_{2n}(\mathbb{R})\,, \K_{2n}\X = \X\K_{2n}\right\}\,.
\end{equation*}
Moreover, using Lemma \ref{LemmaLoewnerEmbedding}, if $\X = \begin{pmatrix} \A & \B \\ -\B & \A \end{pmatrix}$ is a matrix in $\mathscr{P}_{2n}(\mathbb{R})$, then the corresponding matrix $\A + i\B$ is in $\mathscr{P}_{n}(\mathbb{C})$, and the result follows\,.

\end{proof}

\noindent We now give an analogue for $\mathbb{H}$. Every quaternionic number $q \in \mathbb{H}$ can be written as
\begin{equation*}
q = a + bi + cj + dk\,,
\end{equation*}
with $a\,, b\,, c\,, d \in \mathbb{R}$. In particular, using that $k = ij$, we have
\begin{equation*}
q = (a + bi) + (c + di)j\,,
\end{equation*}
i.e. 
\begin{equation}
\mathbb{H} = \mathbb{C} \oplus \mathbb{C}j\,.
\label{EquationQuaternions}
\end{equation}
In particular, it follows from Equation \eqref{EquationQuaternions} that  \begin{equation*}
\Mat_{n}(\mathbb{H}) = \Mat_{n}(\mathbb{C}) \oplus \Mat_{n}(\mathbb{C})j\,.
\end{equation*}

\noindent We denote by $\Psi_{2}: \mathbb{H} \to \Mat_{2}(\mathbb{C})$ the map given by
\begin{equation*}
\Psi_{2}(z_{1} + z_{2}j) = \begin{pmatrix} z_{1} & z_{2} \\ -\overline{z_{2}} & \overline{z_{1}} \end{pmatrix}\,, \qquad \left(z_{1}\,, z_{2} \in \mathbb{C}\right)\,.
\end{equation*}
The map $\Psi_{2}$ is a monomorphism of algebras such that
\begin{equation*}
\Psi_{2}(\overline{q}) = \Psi_{2}(q)^{*}\,, \qquad \left(q \in \mathbb{H}\right)\,.
\end{equation*}
The map $\Psi_{2}$ can be extended to $\Mat_{n}(\mathbb{H})$ by
\begin{equation*}
\Psi_{2}: \Mat_{n}(\mathbb{H}) \mapsto \Mat_{2n}(\mathbb{C})\,, \quad \Z_{1} + \Z_{2}j \mapsto \begin{pmatrix} \Z_{1} & \Z_{2} \\ -\overline{\Z_{2}} & \overline{\Z_{1}} \end{pmatrix}\,, \quad \left(\Z_{1}\,, \Z_{2} \in \Mat_{n}(\mathbb{C})\right)\,,
\end{equation*}
and such that
\begin{equation*}
\Psi_{2}(\Z^{*}) = \Psi_{2}(\Z)^{*}\,, \qquad \left(\Z \in \Mat_{n}(\mathbb{H})\right)\,.
\end{equation*}

\begin{proposition}

We have $\Psi_{2}\left(\mathscr{P}_{n}(\mathbb{H})\right) \subseteq \mathscr{P}_{2n}(\mathbb{C})$. More precisely, we get
\begin{equation*}
\Psi_{2}\left(\mathscr{P}_{n}(\mathbb{H})\right) = \left\{\X \in \mathscr{P}_{2n}(\mathbb{C})\,, \K_{2n}\overline{\X} = \X\K_{2n}\right\}\,.
\end{equation*}

\label{PropositionHtoC}

\end{proposition}

\begin{proof}

See \cite[Appendix~A]{FRANCOMERINO}\,.

\end{proof}

\noindent In particular, it follows from Propositions \ref{PropositionCtoR} and \ref{PropositionHtoC} that for all $n\,, m \in \mathbb{N}^{*}$, we have
\begin{equation*}
\mathscr{P}_{n}(\mathbb{H}) \hookrightarrow \mathscr{P}_{2n}(\mathbb{C})\,, \qquad \mathscr{P}_{m}(\mathbb{C}) \hookrightarrow \mathscr{P}_{2m}(\mathbb{R})\,,
\end{equation*}
i.e.
\begin{equation*}
\mathscr{P}_{n}(\mathbb{H}) \hookrightarrow \mathscr{P}_{2n}(\mathbb{C}) \hookrightarrow \mathscr{P}_{4n}(\mathbb{R})\,.
\end{equation*}

\begin{notation}

We denote by $\widetilde{\mathscr{P}}_{n}(\mathbb{H})$ and $\widetilde{\mathscr{P}}_{m}(\mathbb{C})$ the subsets of $\mathscr{P}_{2n}(\mathbb{C})$ and $\mathscr{P}_{2m}(\mathbb{R})$ respectively given by
\begin{equation*}
\widetilde{\mathscr{P}}_{n}(\mathbb{H}) := \Psi_{2}\left(\mathscr{P}_{n}(\mathbb{H})\right)\,, \qquad \widetilde{\mathscr{P}}_{m}(\mathbb{C}) := \Psi_{1}\left(\mathscr{P}_{m}(\mathbb{C})\right)\,.
\end{equation*}

\end{notation}

\begin{remark}

We denote by $\langle\cdot\,, \cdot\rangle$ the positive Hermitian form on $\mathbb{D}^{n}$ given by
\begin{equation*}
\langle x\,, y\rangle = \sum\limits_{k = 1}^{n} \iota(x_{k})y_{k}\,.
\end{equation*}
For all $\A \in \Mat_{n}(\mathbb{D})$, we have
\begin{equation*}
\langle \A u\,, v\rangle = \langle u\,, \A^{*} v\rangle\,, \qquad \left(u\,, v \in \mathbb{D}^{n}\right)\,,
\end{equation*}
i.e. $\A^{*}$ is the adjoint of $\A$ with respect to the form $\langle\cdot\,, \cdot\rangle$\,.

\noindent We denote by $\U_{\mathbb{D}}$ the subgroup of $\GL_{n}(\mathbb{D})$ given by
\begin{equation*}
\U_{\mathbb{D}} := \left\{g \in \GL_{n}(\mathbb{D})\,, \langle g(x)\,, g(y)\rangle = \langle x\,, y\rangle\,, x\,, y \in \mathbb{D}^{n}\right\}\,.
\end{equation*}
In particular, using the notations of \cite{KNAPP}, we have
\begin{equation*}
\U_{\mathbb{D}} = \begin{cases} \O(n) & \text{ if } \mathbb{D} = \mathbb{R} \\ \U(n) & \text{ if } \mathbb{D} = \mathbb{C} \\ \Sp(n) & \text{ if } \mathbb{D} = \mathbb{H} \end{cases}
\end{equation*}
It is well-known (see \cite{BHATIA} for $\mathbb{D} \in \left\{\mathbb{R}\,, \mathbb{C}\right\}$ and \cite[Appendix~A]{FRANCOMERINO} for $\mathbb{D} = \mathbb{H}$) that for a matrix $\X \in \mathscr{P}_{n}(\mathbb{D})$, we have $\Spec(\X) \subseteq \left(0, \infty\right)$, and there exists a matrix $\U \in \U_{\mathbb{D}}$ and a diagonal matrix $\Lambda = \diag\left(\lambda_{1}\,, \ldots\,, \lambda_{n}\right) \in \Mat_{n}(\mathbb{R})$, with $\lambda_{i} > 0$, such that 
\begin{equation}
\X = \U\Lambda\U^{*}\,.
\label{DecompositionX}
\end{equation}

\begin{remark}

In the decomposition \eqref{DecompositionX} for $\X \in \mathscr{P}_{n}(\mathbb{D})$, the diagonal matrix $\Lambda \in \Mat_{n}(\mathbb{R})$, and the diagonal entries of $\lambda$ are precisely the eigenvalues of $\X$. Indeed, for a positive Hermitian matrix $\X$, we have $\Spec(\X) \subseteq \left(0, \infty\right)$. In the case $\mathbb{D} = \mathbb{H}$, we use right-eigenvalues, i.e. $\lambda \in \Spec(\X)$ if $\A v = v\lambda$ for some non-zero vector $v \in \mathbb{D}^{n}$. As explained in \cite{ZHANG}, if $\X \in \mathscr{P}_{n}(\mathbb{H})$, then $\X$ has $n$ right eigenvalues in $\mathbb{R}$\,.

\end{remark}

\noindent If $f: \left(0, \infty\right) \to \left(0, \infty\right)$ is a Borel function, then $f$ can be extended to a function 
\begin{equation*}
f: \mathscr{P}_{n}(\mathbb{D}) \mapsto \mathscr{P}_{n}(\mathbb{D})
\end{equation*}
given by
\begin{equation*}
f(\X = \U\Lambda\U^{*}) = \U f(\Lambda)\U^{*}\,,
\end{equation*}
with $f\left(\Lambda = \diag\left(\lambda_{1}\,, \ldots\,, \lambda_{n}\right)\right) = \diag\left(f(\lambda_{1})\,, \ldots\,, f(\lambda_{n})\right)$\,.

\label{RemarkSpectral}

\end{remark}

\noindent We finish this section with a lemma\,.

\begin{lemma}

Let $f: \left(0, \infty\right) \to \left(0, \infty\right)$ be a Borel function. Then 
\begin{equation*}
f(\Psi_{1}(\A)) = \Psi_{1}(f(\A))\,, \quad f(\Psi_{2}(\B)) = \Psi_{2}(f(\B))\,, \quad \left(\A \in \mathscr{P}_{n}(\mathbb{C})\,, \B \in \mathscr{P}_{n}(\mathbb{H})\right)\,.
\end{equation*}

\label{LemmaFunctionalCalculus}

\end{lemma}

\begin{proof}

We prove the statement for $\Psi_{1}$. The proof for $\Psi_{2}$ is identical. Let $\A \in \mathscr{P}_{n}(\mathbb{C})$. Since $\A$ is Hermitian and positive, it follows from Remark \ref{RemarkSpectral} that there exists $\lambda_{1}\,, \ldots\,, \lambda_{r}> 0 $ and pairwise orthogonal projections $\P_{1}\,, \ldots\,, \P_{r}$ such that
\begin{equation*}
\A = \sum\limits_{j = 1}^{r} \lambda_{j}\P_{j}\,.
\end{equation*}
Since $\Psi_{1}$ is an algebra homomorphism and preserves adjoints (see Equation \eqref{AdjointPsiOne}), we have
\begin{equation*}
\Psi_{1}(\A) = \sum\limits_{j = 1}^{r} \lambda_{j}\Psi_{1}(\P_{j})\,.
\end{equation*}
Moreover, the matrices $\Psi_{1}(\P_{j})$ are again pairwise orthogonal projections, because
\begin{equation*}
\Psi_{1}(\P_{j})^{2} = \Psi_{1}(\P_{j}^{2}) = \Psi_{1}(\P_{j})\,,
\end{equation*}
and, for $i \neq j$, we get
\begin{equation*}
\Psi_{1}(\P_{i})\Psi_{1}(\P_{j}) = \Psi_{1}(\P_{i}\P_{j}) = 0\,.
\end{equation*}
Thus this is the spectral decomposition of $\Psi_{1}(\A)$. Therefore, by classical functional calculus,
\begin{equation*}
f(\Psi_{1}(\A)) = \sum\limits_{j=1}^{r} f(\lambda_{j})\Psi_{1}(\P_{j})\,.
\end{equation*}
On the other hand,
\begin{equation*}
f(\A) = \sum\limits_{j=1}^{r} f(\lambda_{j})\P_{j}\,.
\end{equation*}
Applying $\Psi_{1}$ gives
\begin{equation*}
\Psi_{1}(f(\A)) = \sum_{j=1}^{r} f(\lambda_{j})\Psi_{1}(\P_{j})\,.
\end{equation*}
Finally
\begin{equation*}
f(\Psi_{1}(\A)) = \Psi_{1}(f(\A))\,.
\end{equation*}

\end{proof}

\noindent The next lemma shows a compatibility between the Loewner order and the maps $\Psi_{i}$\,.

\begin{lemma}
 
The maps $\Psi_{1}$ and $\Psi_{2}$ are order embeddings. More precisely,
\begin{enumerate}
\item for all $\A\,, \B \in \mathfrak{p}_{n}(\mathbb{C})$, we have $\A \leq \B$ if and only if $\Psi_{1}(\A) \leq \Psi_{1}(\B)$\,,
\item for all $\C\,, \D \in \mathfrak{p}_{m}(\mathbb{H})$, we have $\C \leq \D$ if and only if $\Psi_{2}(\C) \leq \Psi_{2}(\D)$\,.
\end{enumerate}
In particular,
\begin{equation*}
\Psi_{1}\left(\mathscr{P}^{0}_{n}(\mathbb{C})\right) = \left\{\X \in \mathscr{P}^{0}_{2n}(\mathbb{R})\,, \K_{2n}\X = \X\K_{2n}\right\}\,.
\end{equation*}
 
\label{LemmaLoewnerEmbedding}
 
\end{lemma}
 
\begin{proof}
 
We prove it for $\Psi_{1}$ (the proof for $\Psi_{2}$ is identical, using the spectral theorem over $\mathbb{H}$ (see \cite{ZHANG}) and the fact that real scalars are central in $\mathbb{H}$).  Since $\Psi_{1}$ is linear, $\Psi_{1}(\B)-\Psi_{1}(\A) = \Psi_{1}(\B-\A)$, so it suffices to show that for every $\H \in \mathfrak{p}_{n}(\mathbb{C})$,
\begin{equation*}
\H \geq 0 \qquad \Longleftrightarrow \qquad \Psi_{1}(\H) \geq 0\,.
\end{equation*}
By the spectral theorem for Hermitian matrices over $\mathbb{C}$ (see \cite{BHATIA}), there exist $\U \in \U_{\mathbb{C}}$ and a real diagonal matrix $\Lambda = \diag(\lambda_{1}\,, \ldots\,, \lambda_{n})$ such that $\H = \U\Lambda\U^{*}$. For $x \in \mathbb{C}^{n}$ and $y = \U^{*}x$, we have
\begin{equation*}
x^{*}\H x = y^{*}\Lambda y = \sum\limits_{i=1}^{n} \lambda_{i}\left|y_{i}\right|^{2}\,,
\end{equation*}
since each $\lambda_{i}$ is real, hence central in $\mathbb{C}$. As $y$ ranges over $\mathbb{C}^{n}$ when $x$ does, $\H \geq 0$ if and only if $\lambda_{i} \geq 0$ for all $i$.
 
\noindent On the other hand, $\Psi_{1}$ is a unital algebra homomorphism compatible with adjoints, so $\Psi_{1}(\U)\Psi_{1}(\U)^{*} = \Psi_{1}(\U\U^{*}) = \Id_{2n}$, i.e. $\Psi_{1}(\U)$ is orthogonal. Moreover, since $\Lambda$ is real, $\Psi_{1}(\Lambda) = \diag(\Lambda\,, \Lambda)$. Therefore
\begin{equation*}
\Psi_{1}(\H) = \Psi_{1}(\U)\diag(\Lambda\,, \Lambda)\Psi_{1}(\U)^{*}\,,
\end{equation*}
so the eigenvalues of $\Psi_{1}(\H)$ are those of $\H$, each repeated twice. Hence $\Psi_{1}(\H) \geq 0$ if and only if $\lambda_{i} \geq 0$ for all $i$, i.e. if and only if $\H \geq 0$.
 
\end{proof}

\section{Kubo-Ando means of finite order over $\mathbb{D}$}

In this section, we study Kubo-Ando means of finite order, that is,
means defined on the cone of positive definite matrices of a fixed
size, which are not necessarily associated with operator monotone
functions (see \cite{KUBOANDO}). Some results in this direction have already appeared in the literature (see, for example, \cite{BENDATSHERMAN, HANSENJITOMIYAMA, HANSENTOMIYAMA, HIAISANO, HOAHOOSAKA}). However, to the best of our knowledge, a general correspondence theorem between Kubo-Ando means of order $n$ and a special class of normalized matrix monotone functions of order $n$ has not been explicitly established in the literature. We begin by proving this correspondence.

\noindent We recall that a function
\begin{equation*}
f: \left(0, \infty\right) \longrightarrow \left(0, \infty\right)
\end{equation*}
is said to be matrix monotone of order $n$ over $\mathbb{D}$ if, for all $\A, \B \in \mathscr{P}_{n}(\mathbb{D})$,
\begin{equation*}
\A \leq \B \qquad \Longrightarrow \qquad f(\A) \leq f(\B)\,.
\end{equation*}

\begin{remark}

It is known (see \cite{HANSENJITOMIYAMA}), that if $f$ is matrix monotone of order $n$, $f$ is matrix monotone of order $m$ for all $m \leq n$. However, $f$ is not, in general, matrix monotone of order $n+1$ (see \cite{HANSENJITOMIYAMA}).

\end{remark}

\begin{definition}

Let $\mathbb{D }\in \{\mathbb{R},\mathbb{C}, \mathbb{H}\}$. A Kubo-Ando mean of order $n$ on $\mathscr{P}_{n}(\mathbb{D})$ is a binary operation
\begin{equation*}
\sigma: \mathscr{P}_{n}(\mathbb{D}) \times \mathscr{P}_{n}(\mathbb{D}) \longrightarrow \mathscr{P}_{n}(\mathbb{D})
\end{equation*}
satisfying the following properties:
\begin{enumerate}
\item For all $\A, \B, \C, \D \in \mathscr{P}_{n}(\mathbb{D})$,
\begin{equation*}
\A \leq \C \quad \text{ and } \quad \B \leq \D \qquad \Longrightarrow \qquad \A \sigma \B \leq \C \sigma \D\,.
\end{equation*}
\item For all $\A, \B \in \mathscr{P}_{n}(\mathbb{D})$ and every $\C \in \GL_{n}(\mathbb{D})$,
\begin{equation*}
\left(\C^{*}\A\C\right) \sigma \left(\C^{*}\B\C\right) = \C^{*}\left(\A \sigma \B\right)\C\,.
\end{equation*}
\item The map
\begin{equation*}
\left(\A, \B\right) \longmapsto \A \sigma \B
\end{equation*}
is continuous on $\mathscr{P}_{n}(\mathbb{D}) \times \mathscr{P}_{n}(\mathbb{D})$.
\item We have $\Id_{n} \sigma \Id_{n} = \Id_{n}$.
\end{enumerate}

\end{definition}

\noindent We first discuss the representing function associated with a Kubo-Ando mean of finite order.

\begin{theo}

Let $\sigma$ be a Kubo-Ando mean of order $n$ on $\mathscr{P}_{n}(\mathbb{D})$. Then there exists a unique function
\begin{equation*}
f: \left(0, \infty\right) \longrightarrow \left(0, \infty\right)
\end{equation*}
such that for every $t>0$,
\begin{equation}
\Id_{n} \sigma \left(t\Id_{n}\right) = f(t)\Id_{n}
\label{ScalarRepresentingFunction}
\end{equation}
and with
\begin{equation}
\A \sigma \B = \A^{\frac{1}{2}}f\left(\A^{-\frac{1}{2}}\B\A^{-\frac{1}{2}}\right)\A^{\frac{1}{2}}
\label{FiniteOrderRepresentation}
\end{equation}
for all $\A, \B \in \mathscr{P}_{n}(\mathbb{D})$. Moreover, $f(1)=1$.

\label{TheoremRepresentation}

\end{theo}

\begin{proof}

For every $t>0$, congruence invariance implies that $\Id_{n} \sigma (t\Id_n)$ commutes with every unitary matrix in $\U_{\mathbb{D}}$. Since the natural representation of $\U_{\mathbb{D}}$ on $\mathbb{D}^{n}$ is irreducible, Schur's lemma implies that $\Id_{n} \sigma (t\Id_{n})$ is a real scalar multiple of the identity. Therefore, there exists a unique function $f: \left(0, \infty\right) \to \left(0, \infty\right)$ such that
\begin{equation*}
\Id_{n} \sigma (t\Id_{n}) = f(t)\Id_{n}\,.
\end{equation*}
Let $\D = \diag(d_{1}, \dots, d_{n}) \in \mathscr{P}_{n}(\mathbb{D})$. For a diagonal sign matrix $\S = \diag(\pm1,\dots,\pm1)$, congruence invariance gives
\begin{equation*}
\S(\Id_{n}\sigma \D) \S = \Id_{n} \sigma \D\,,
\end{equation*}
so $\Id_{n} \sigma \D$ commutes with every such $\S$ and is therefore diagonal. Write $g_{i}(\D)$ for its $i$-th diagonal entry. One can see that $g_{i}(\D)$ depends only on $d_{i}$. Let $\D' = \diag(d'_{1}, \dots, d'_{n})$ with $d'_{i} = d_{i}$. Choose $c \in (0,1]$ with $c^{2}d_{j} \leq d'_{j}$ for all $j \neq i$, and set $\C = \diag(c_1,\dots,c_n)$ with $c_i=1$ and $c_j=c$ for $j\ne i$. By congruence invariance, 
\begin{equation*}
\C(\Id_{n} \sigma \D) \C = \C^{2}\sigma(\C\D\C)
\end{equation*}
whose $i$-th diagonal entry is $g_{i}(\D)$. Since $\C^{2}\leq\Id_n$ and $\C\D\C\leq\D'$, monotonicity gives 
\begin{equation*}
\C^{2} \sigma \left(\C\D\C\right) \leq \Id_{n} \sigma \D'\,,
\end{equation*}
and comparing $i$-th diagonal entries yields $g_i(\D) \le g_i(\D')$. Exchanging the roles of $\D$ and $\D'$ gives equality. Taking $\D' = d_i\Id_n$ shows $g_i(\D)=f(d_i)$, i.e. $\Id_{n}\sigma\D=f(\D)$. The case of a general matrix follows from the spectral theorem and congruence invariance.

\noindent Let $\A, \B \in \mathscr{P}_{n}(\mathbb{D})$, and let $\X = \A^{-\frac{1}{2}}\B\A^{-\frac{1}{2}}$. By congruence invariance,
\begin{equation*}
\A \sigma \B = \A^{\frac{1}{2}}\left(\Id_{n} \sigma \A^{-\frac{1}{2}}\B\A^{-\frac{1}{2}}\right)\A^{\frac{1}{2}} = \A^{\frac{1}{2}}f(\X)\A^{\frac{1}{2}}\,.
\end{equation*}
The normalization condition gives
\begin{equation*}
\Id_{n} = \Id_{n} \sigma \Id_{n} = f(1)\Id_{n}\,,
\end{equation*}
and hence $f(1)=1$. Finally, if two functions $f$ and $g$ satisfy
\eqref{ScalarRepresentingFunction}, then
\begin{equation*}
f(t)\Id_{n} = \Id_{n} \sigma (t\Id_{n}) = g(t)\Id_{n}
\end{equation*}
for every $t>0$. Then $f=g$.

\end{proof}

\noindent The preceding theorem allows us to establish additional properties of the representing function.

\begin{proposition}

Let $\sigma$ be a Kubo-Ando mean on $\mathscr{P}_{n}(\mathbb{D})$ admitting a representing function $f$. Then both functions
\begin{equation*}
f(t) \qquad \text{and} \qquad f^{\top}(t) = tf(t^{-1})
\end{equation*}
are matrix monotone of order $n$ over $\mathbb{D}$.

\label{PropositionFiniteOrderFFtop}

\end{proposition}

\begin{proof}

Let $\A, \B \in \mathscr{P}_{n}(\mathbb{D})$ such that $\A \leq \B$. Since $\sigma$ is monotone in its second variable,
\begin{equation*}
\Id_{n} \sigma \A \leq \Id_{n} \sigma \B\,.
\end{equation*}
Using the representation formula (i.e. $f$), we obtain
\begin{equation*}
f(\A)\leq f(\B).
\end{equation*}
Therefore, $f$ is matrix monotone of order $n$. 

\noindent For every $\X \in \mathscr{P}_{n}(\mathbb{D})$, functional calculus gives
\begin{equation*}
\X \sigma \Id_{n} = \X^{\frac{1}{2}}f(\X^{-1})\X^{\frac{1}{2}} = \X f(\X^{-1})= f^{\top}(\X)\,.
\end{equation*}
By monotonicity of $\sigma$ in its first variable,
\begin{equation*}
\A \sigma \Id_{n} \leq \B \sigma \Id_{n}\,.
\end{equation*}
Therefore
\begin{equation*}
f^{\top}(\A) \leq f^{\top}(\B)\,.
\end{equation*}
Thus $f^{\top}$ is also matrix monotone of order $n$.

\end{proof}

\noindent We now introduce the class
\begin{equation*}
\mathcal{F}_{n}(\mathbb{D}) = \left\{f: (0,\infty) \to (0,\infty)\,, f(1)=1\,, f \text{ and } f^{\top} \text{ are matrix monotone of order n}\right\}
\end{equation*}
Let $\mathcal{M}_{n}(\mathbb{D})$ denote the class of Kubo-Ando means of finite order $n$.

\begin{remark}

It is known (see \cite{HIAISANO}) that if $f$ is matrix monotone of order $2n$, then $f^{\top}$ is matrix monotone of order $n$. However, in general, the two functions have different orders\,.

\end{remark}

\begin{proposition}

The correspondence
\begin{equation*}
\Phi: \mathcal{M}_{n}(\mathbb{D}) \longrightarrow \mathcal{F}_{n}(\mathbb{D})\,, \qquad  \sigma \longmapsto f
\end{equation*}
is well-defined and injective\,.

\end{proposition}

\begin{proof}

Proposition \ref{PropositionFiniteOrderFFtop} shows that $\Phi$ is well-defined. Suppose that two means $\sigma_{1}$ and $\sigma_{2}$ have the same representing function $f$. For all $\A, \B \in \mathscr{P}_{n}(\mathbb{D})$, we have
\begin{equation*}
\A \sigma_{1} \B = \A^{\frac{1}{2}}f(\A^{-\frac{1}{2}}\B\A^{-\frac{1}{2}})\A^{\frac{1}{2}} = \A \sigma_{2}\B\,.
\end{equation*}
Hence $\sigma_{1} = \sigma_{2}$, and $\Phi$ is injective.

\end{proof}

\noindent We next prove that every function in $\mathcal{F}_{n}(\mathbb{D})$ determines a mean.

\begin{proposition}

Let $f \in \mathcal{F}_{n}(\mathbb{D})$. Define
\begin{equation}
\A \sigma_{f} \B = \A^{\frac{1}{2}}f(\A^{-\frac{1}{2}}\B\A^{-\frac{1}{2}})\A^{\frac{1}{2}}\,.
\label{MeanAssociatedWithFunction}
\end{equation}
Then $\sigma_{f}$ is a Kubo-Ando mean of finite order on $\mathscr{P}_{n}(\mathbb{D})$.

\end{proposition}

\begin{proof}

Since $f$ is positive-valued, the operation $\sigma_{f}$ takes values in $\mathscr{P}_{n}(\mathbb{D})$. Moreover, $f(1)=1$ implies
\begin{equation*}
\Id_{n} \sigma_{f} \Id_{n} = \Id_{n}\,.
\end{equation*}
We first show that $f$ is continuous on $\left(0, \infty\right)$. Matrix monotonicity of order $n$ implies matrix monotonicity of order $1$. Hence both $f$ and $f^{\top}$ are nondecreasing on $\left(0, \infty\right)$, and the latter means that $s \longmapsto f(s)/s = f^{\top}(s^{-1})$ is nonincreasing. Let $s_{0} > 0$. Since $f$ is nondecreasing, the one-sided limits $f(s_{0}^{\pm})$ exist and $f(s_{0}^{-}) \leq f(s_{0}) \leq f(s_{0}^{+})$. For $0 < s < s_{0} < s'$, we have 
\begin{equation*}
\frac{f(s')}{s'} \leq \frac{f(s)}{s}\,,
\end{equation*}
and letting $s \to s_{0}$ with $s < s_{0}$, and $s' \to s_{0}$, with $s_{0} < s'$, gives $f(s_{0}^{+}) \leq f(s_{0}^{-})$. Therefore 
\begin{equation*}
f(s_{0}^{-}) = f(s_{0}) = f(s_{0}^{+})\,,
\end{equation*}
and $f$ is continuous at $s_{0}$. For a continuous function $g$ on $\left(0, \infty\right)$, the map $\X \longmapsto g(\X)$ is continuous on $\mathscr{P}_{n}(\mathbb{D})$. This is classical for $\mathbb{D} \in \left\{\mathbb{R}\,, \mathbb{C}\right\}$ (see \cite{BHATIA2}), and the case $\mathbb{D} = \mathbb{H}$ follows from the complex case via $\Psi_{2}$ and Lemma \ref{LemmaFunctionalCalculus}. Applying this to $g = f$ and $g(t) = t^{\pm\frac{1}{2}}$ shows that $\sigma_{f}$ is continuous on $\mathscr{P}_{n}(\mathbb{D}) \times \mathscr{P}_{n}(\mathbb{D})$.

\noindent We now establish congruence invariance.
Let $\C \in \mathrm{GL}_{n}(\mathbb{D})$, and let $\X = \A^{-\frac{1}{2}}\B\A^{-\frac{1}{2}}$. The matrix
\begin{equation*}
\U = \left(\C^{*}\A\C\right)^{-\frac{1}{2}}\C^{*}\A^{\frac{1}{2}}
\end{equation*}
is unitary, and
\begin{equation*}
\left(\C^{*}\A\C\right)^{-\frac{1}{2}}\C^{*}\B\C\left(\C^{*}\A\C\right)^{-\frac{1}{2}} = \U\X\U^{*}\,.
\end{equation*}
By unitary equivariance of functional calculus,
\begin{equation*}
f(\U\X\U^{*}) = \U f(\X)\U^{*}\,.
\end{equation*}
Substituting into \eqref{MeanAssociatedWithFunction}, we obtain
\begin{equation*}
\left(\C^{*}\A\C\right) \sigma_{f} \left(\C^{*}\B\C\right) = \C^{*}\left(\A \sigma_{f} \B\right)\C\,.
\end{equation*}
We now prove monotonicity. Suppose that $\B \leq \D$. Then
\begin{equation*}
\A^{-\frac{1}{2}}\B\A^{-\frac{1}{2}} \leq \A^{-\frac{1}{2}}\D\A^{-\frac{1}{2}}\,.
\end{equation*}
Since $f$ is matrix monotone of order $n$,
\begin{equation*}
\A \sigma_{f} \B \leq \A \sigma_{f} \D\,.
\end{equation*}
Thus $\sigma_{f}$ is monotone in its second variable. To prove monotonicity in the first variable, we use the function $f^{\top}$. We get
\begin{equation}
\A \sigma_{f} \B = \B \sigma_{f^{\top}} \A\,.
\label{TransposePerspectiveIdentity}
\end{equation}
Since $f^{\top}$ is matrix monotone of order $n$, the operation $\sigma_{f^{\top}}$ is monotone in its second variable. Therefore, if $\A \leq \C$, then
\begin{equation*}
\A \sigma_{f} \B = \B \sigma_{f^{\top}} \A \leq \B \sigma_{f^{\top}} \C = \C \sigma_{f} \B\,.
\end{equation*}
Therefore, $\sigma_{f}$ is monotone in both variables, and hence jointly monotone.

\end{proof}

\noindent Combining the preceding results, we obtain the following theorem\,.

\begin{theo}

The map
\begin{equation*}
\Phi: \mathcal{M}_{n}(\mathbb{D}) \longrightarrow \mathcal{F}_{n}(\mathbb{D})\,, \qquad \sigma \longmapsto f\,,
\end{equation*}
is bijective.

\end{theo}

\noindent A second notion, which will be particularly important in the present paper, consists in restricting the matrices under consideration. Let $\mathscr{S} \subseteq \mathscr{P}_{n}(\mathbb{D})$ with $\Id_{n} \in \mathscr{S}$, and let $\G$ be a subgroup of $\GL_{n}(\mathbb{D})$  acting transitively by congruence on $\mathscr{S}$.

\begin{definition}

A Kubo-Ando mean of order $n$ on $\mathscr{S}$ is a binary operation
\begin{equation*}
\sigma: \mathscr{S} \times \mathscr{S} \longrightarrow \mathscr{S}
\end{equation*}
satisfying the following properties:
\begin{enumerate}
\item For all $\A, \B, \C, \D \in \mathscr{S}$,
\begin{equation*}
\A \leq \C \quad \text{ and } \quad \B \leq \D \qquad \Longrightarrow \qquad \A \sigma \B \leq \C \sigma \D\,.
\end{equation*}
\item For all $\A, \B \in \mathscr{S}$ and every $\C \in \G$,
\begin{equation*}
\left(\C^{*}\A\C\right) \sigma \left(\C^{*}\B\C\right) = \C^{*}\left(\A \sigma \B\right)\C\,.
\end{equation*}
\item The map
\begin{equation*}
\left(\A, \B\right) \longmapsto \A \sigma \B
\end{equation*}
is continuous on $\mathscr{S} \times \mathscr{S}$.
\item We have $\Id_{n} \sigma \Id_{n} = \Id_{n}$.
\end{enumerate}

\end{definition}

\medskip

\noindent Similarly, we say that a function $f: \left(0, \infty\right) \longrightarrow \left(0, \infty\right)$ is matrix monotone on $\mathscr{S}$ if, for all $\A, \B \in \mathscr{S}$,
\begin{equation*}
\A \leq \B \qquad \Longrightarrow \qquad f(\A) \leq f(\B)\,.
\end{equation*}
Clearly, a function which is matrix monotone on $\mathscr{P}_{n}(\mathbb{D})$ is matrix monotone on every subset $\mathscr{S}\subseteq\mathscr{P}_{n}(\mathbb{D})$, while the converse need not hold. In particular, this notion allows us to consider distinguished subsets, or subcones, on which matrix monotonicity is preserved even when the function is not matrix monotone on the whole cone.

\medskip

\noindent These two points of view will both appear below. The first keeps the entire cone $\mathscr{P}_{n}(\mathbb{D})$ but fixes the order of matrix monotonicity, while the second fixes a distinguished subset of an ambient cone. The latter is especially natural for the embedded cones $\widetilde{\mathscr{P}}_{n}(\mathbb{C})$ and $\widetilde{\mathscr{P}}_{m}(\mathbb{H})$ introduced in the previous section.

\begin{remark}

One of the main tools for studying matrix monotonicity of finite order is the Loewner matrix. Assume that $f$ is $\mathscr{C}^{1}$. For $x_{1}, \ldots, x_{n} \in \left(0, +\infty\right)$, the Loewner matrix associated with $f$ is defined by
\begin{equation*}
\L_{f}(x_{1},\ldots,x_{n}) = \left(f^{[1]}(x_{i}, x_{j})\right)_{1 \leq i, j \leq n}\,,
\end{equation*}
where
\begin{equation*}
f^{[1]}(x, y) = \begin{cases} \displaystyle\frac{f(x)-f(y)}{x-y} & \text{ if } x\neq y \\ f'(x) & \text{ if } x=y \end{cases}
\end{equation*}
As explained in \cite{BHATIA2}, $f$ is matrix monotone of order $n$ if and only if
\begin{equation*}
\L_{f}(x_{1}\,,\ldots\,,x_{n}) \geq 0
\end{equation*}
for every $x_{1}, \ldots, x_{n} \in \left(0, +\infty\right)$. Thus, the positivity of Loewner matrices provides a useful criterion for establishing matrix monotonicity of a prescribed order.

\label{RemarkLoewner}

\end{remark}

\begin{lemma}

Let $\mathbb{D}\in\{\mathbb{R}\,,\mathbb{C}\,,\mathbb{H}\}$ and let $f: \left(0, +\infty\right) \longrightarrow \mathbb{R}$. For a fixed $n \geq 1$, matrix monotonicity of $f$ of order $n$ does not depend on the choice of the division algebra $\mathbb{D}$. More precisely, the following conditions are equivalent:
\begin{enumerate}
\item $f$ is matrix monotone on $\mathscr{P}_{n}(\mathbb{R})$\,,
\item $f$ is matrix monotone on $\mathscr{P}_{n}(\mathbb{C})$\,,
\item $f$ is matrix monotone on $\mathscr{P}_{n}(\mathbb{H})$\,.
\end{enumerate}

\label{LemmaEquivalenceRCH}

\end{lemma}

\begin{proof}

It is enough here to prove the equivalence between the real and complex
cases. The equivalence between the complex and quaternionic cases is
established in Appendix~\ref{AppendixA}. Assume first that $f$ is matrix monotone on $\mathscr{P}_{n}(\mathbb{C})$. Since $\mathscr{P}_{n}(\mathbb{R}) \subseteq \mathscr{P}_{n}(\mathbb{C})$, the restriction of $f$ to real symmetric positive definite matrices is matrix monotone.

\noindent Conversely, assume that $f$ is matrix monotone on $\mathscr{P}_{n}(\mathbb{R})$. We may assume $n \geq 2$ (for $n = 1$ all three conditions say that $f$ is nondecreasing), so that $f$ is continuously differentiable on $\left(0, \infty\right)$ (see \cite{HIAI2010}), and the necessity part of the Loewner criterion in Remark~\ref{RemarkLoewner} holds over $\mathbb{R}$, since its proof only uses perturbations in real rank-one directions.

\noindent Using Remark \ref{RemarkLoewner}, for every $x_{1}, \ldots, x_{n} > 0$, the corresponding Loewner matrix $\L_{f}(x_{1}\,,\ldots\,,x_{n})$ is positive semidefinite. Since the Loewner matrix depends only on the scalar function $f$ and on the points $x_{1}, \ldots, x_{n}$, the same criterion implies that $f$ is matrix monotone on $\mathscr{P}_{n}(\mathbb{C})$.

\end{proof}

\begin{remark}

Let $\sigma_{\mathbb{C}}$ be a Kubo-Ando mean on $\mathscr{P}_{n}(\mathbb{C})$. Then it defines a map 
\begin{equation*}
\widetilde{\sigma}_{\mathbb{C}}: \widetilde{\mathscr{P}}_{n}(\mathbb{C}) \times \widetilde{\mathscr{P}}_{n}(\mathbb{C}) \to \widetilde{\mathscr{P}}_{n}(\mathbb{C})
\end{equation*}
given by
\begin{equation*}
\A \widetilde{\sigma}_{\mathbb{C}} \B = \Psi_{1}\left(\Psi^{-1}_{1}(\A) \sigma_{\mathbb{C}} \Psi^{-1}_{1}(\B)\right)\,, \qquad \left(\A\,, \B \in \widetilde{\mathscr{P}}_{n}(\mathbb{C})\right)\,.
\end{equation*}
Moreover, for all $g \in \G := \Psi_{1}(\GL_{n}(\mathbb{C}))$, we have
\begin{equation*}
g\left(\A \widetilde{\sigma}_{\mathbb{C}}\, \B\right)g^{t} = \left(g\A g^{t}\right) \widetilde{\sigma}_{\mathbb{C}} \left(g\B g^{t}\right)\,.
\end{equation*}
In particular, $\widetilde{\sigma}_{\mathbb{C}}$ is a Kubo-Ando mean of order $n$ on $\widetilde{\mathscr{P}}_{n}(\mathbb{C})$.

\label{RemarkTransferCtoR}

\end{remark}

\begin{proposition}

Let $\sigma_{1}$ and $\sigma_{2}$ be two Kubo-Ando means of order $2n$ on $\mathscr{P}_{2n}(\mathbb{R})$ such that $\sigma_{1} = \sigma_{2}$ on $\widetilde{\mathscr{P}}_{n}(\mathbb{C})$. Then $\sigma_{1} = \sigma_{2}$\,.

\label{InjectivityMeansCtoR}

\end{proposition}

\begin{proof}

We denote by $f_{1}$ and $f_{2}$ matrix monotone functions of order $2n$ coming from Theorem \ref{TheoremRepresentation}. For all $x > 0$, $x\Id_{n} \in \mathscr{P}_{n}(\mathbb{C})$, and $\Psi_{1}(x\Id_{n}) = x\Id_{2n} \in \widetilde{\mathscr{P}}_{n}(\mathbb{C})$. Using that
\begin{equation*}
\Id_{2n} \sigma_{1} \left(x\Id_{2n}\right) = f_{1}(x)\Id_{2n} \qquad \text{ and } \qquad \Id_{2n} \sigma_{2} \left(x\Id_{2n}\right) = f_{2}(x)\Id_{2n}
\end{equation*}
and that $\sigma_{1} = \sigma_{2}$ on $\widetilde{\mathscr{P}}_{n}(\mathbb{C})$, it follows that 
\begin{equation*}
f_{1}(x) = f_{2}(x)\,, \qquad \left(x \in \left(0, \infty\right)\right)\,.
\end{equation*}
Therefore $\sigma_{1} = \sigma_{2}$\,.

\end{proof}

\begin{remark}

Let $\sigma_{\mathbb{C}}$ be a Kubo-Ando mean of order $n$ on $\mathscr{P}_{n}(\mathbb{C})$, and let $f$ be its representing function. As explained in Remark \ref{RemarkTransferCtoR}, $\sigma_{\mathbb{C}}$ induces a mean $\widetilde{\sigma}_{\mathbb{C}}$ on $\widetilde{\mathscr{P}}_{n}(\mathbb{C})$. Moreover, $f$ is matrix monotone on
\begin{equation*}
\widetilde{\mathscr{P}}_{n}(\mathbb{C})\subseteq\mathscr{P}_{2n}(\mathbb{R})\,.
\end{equation*}
However, $f$ need not be matrix monotone of order $2n$. Therefore, $\widetilde{\sigma}_{\mathbb{C}}$ does not necessarily extend to a Kubo-Ando mean on $\mathscr{P}_{2n}(\mathbb{R})$. 

\end{remark}

\begin{lemma}

Let $\sigma \in\mathcal{M}_{n}(\mathbb{C})$ have representing function $f$. Then $\widetilde{\sigma}$ extends to a Kubo-Ando mean of order $2n$ on $\mathscr{P}_{2n}(\mathbb{R})$ if and only if $f \in \mathcal{F}_{2n}(\mathbb{R})$. In that case, the extension is unique and equals $\sigma_{f}$.

\label{LemmaExtensionCtoR}

\end{lemma}

\begin{proof}
    
If $f \in \mathcal{F}_{2n}(\mathbb{R})$, then $\sigma_{f}$ restricts to $\widetilde{\sigma}$ by Lemma \ref{LemmaFunctionalCalculus}, applied to $f$ and to $t \to t^{\pm\frac{1}{2}}$. Conversely, if $\tau$ extends $\widetilde{\sigma}$, then 
\begin{equation*}
\Id_{2n} \tau \left(x\Id_{2n}\right) = \Psi_{1}\left(\Id_{n} \sigma x\Id_{n}\right) = f(x)\Id_{2n}\,.
\end{equation*}
So $\tau$ has representing function $f$, and $f \in \mathcal{F}_{2n}(\mathbb{R})$ by Proposition \ref{PropositionFiniteOrderFFtop}. Moreover, the uniqueness follows from Proposition \ref{InjectivityMeansCtoR}.

\end{proof}

\begin{example}
 
The condition in Lemma \ref{LemmaExtensionCtoR} is not automatic. Indeed, take $n = 1$, so that $\mathscr{P}_{1}(\mathbb{C}) = \left(0, \infty\right)$ and $\widetilde{\mathscr{P}}_{1}(\mathbb{C}) = \left\{a\Id_{2}\,, a > 0\right\}$. Let $\sigma$ be the mean given by
\begin{equation*}
a \sigma b = \min(a\,, b)\,, \qquad \left(a\,, b \in \left(0, \infty\right)\right)\,,
\end{equation*}
with representing function $f(t) = \min(1\,, t)$. Since $f^{\top}(t) = t\min(1\,, t^{-1}) = \min(t\,, 1) = f(t)$, and $f$ is nondecreasing, we have $f \in \mathcal{F}_{1}(\mathbb{C})$. However, for $x_{1} = \frac{1}{2}$ and $x_{2} = 2$, we get
\begin{equation*}
\L_{f}(x_{1}\,, x_{2}) = \begin{pmatrix} 1 & \frac{1}{3} \\ \frac{1}{3} & 0 \end{pmatrix}\,, \qquad \det \L_{f}(x_{1}\,, x_{2}) = -\frac{1}{9} < 0\,,
\end{equation*}
so $f$ is not matrix monotone of order $2$, and the induced mean on $\widetilde{\mathscr{P}}_{1}(\mathbb{C})$ does not extend to a Kubo-Ando mean on $\mathscr{P}_{2}(\mathbb{R})$.
 
\end{example}

\noindent We can easily see that we have an analogue for quaternionic Hermitian matrices. Indeed, any Kubo-Ando mean $\sigma_{\mathbb{H}}$ of order $m$ on $\mathscr{P}_{m}(\mathbb{H})$ defines a map 
\begin{equation*}
\widetilde{\sigma}_{\mathbb{H}}: \widetilde{\mathscr{P}}_{m}(\mathbb{H}) \times \widetilde{\mathscr{P}}_{m}(\mathbb{H}) \to \widetilde{\mathscr{P}}_{m}(\mathbb{H})
\end{equation*}
given by
\begin{equation*}
\A \widetilde{\sigma}_{\mathbb{H}} \B = \Psi_{2}\left(\Psi^{-1}_{2}(\A) \sigma_{\mathbb{H}} \Psi^{-1}_{2}(\B)\right)\,, \qquad \left(\A\,, \B \in \widetilde{\mathscr{P}}_{m}(\mathbb{H})\right)\,.
\end{equation*}
The next proposition is an analogue of Proposition \ref{InjectivityMeansCtoR} for $\mathbb{H}$\,.

\begin{proposition}

Let $\sigma_{1}$ and $\sigma_{2}$ be two Kubo-Ando means of order $2m$ on $\mathscr{P}_{2m}(\mathbb{C})$ such that $\sigma_{1} = \sigma_{2}$ on $\widetilde{\mathscr{P}}_{m}(\mathbb{H})$. Then $\sigma_{1} = \sigma_{2}$\,.

\label{InjectivityMeansHtoC}

\end{proposition}

\begin{proof}

The proof is identical to that of Proposition \ref{InjectivityMeansCtoR} with $\Psi_{1}$ replaced by $\Psi_{2}$.
    
\end{proof}

\noindent Similarly, the same argument, with $\Psi_{2}$ and Proposition \ref{InjectivityMeansHtoC}, shows that for $\sigma \in \mathcal{M}_{m}(\mathbb{H})$ with representing function $f$, the induced mean $\widetilde{\sigma}_{\mathbb{H}}$ extends to a Kubo-Ando mean of order $2m$ on $\mathscr{P}_{2m}(\mathbb{C})$ if and only if $f \in \mathcal{F}_{2m}(\mathbb{C})$, in which case the extension is unique and equals $\sigma_{f}$.

\begin{remark}

The preceding results naturally raise a maximality question. Let $f$ be matrix monotone of order $n$, but not necessarily of order $2n$. The canonical embedding $\Psi_{1}$ ensures that $f$ preserves the Loewner order on the embedded cone
\begin{equation*}
\widetilde{\mathscr{P}}_{n}(\mathbb{C}) \subset \mathscr{P}_{2n}(\mathbb{R})\,,
\end{equation*}
whereas monotonicity on the whole cone $\mathscr{P}_{2n}(\mathbb{R})$ is equivalent to matrix monotonicity of order $2n$. It is therefore natural to ask whether $\widetilde{\mathscr{P}}_{n}(\mathbb{C})$ is maximal among the subcones of $\mathscr{P}_{2n}(\mathbb{R})$ on which every matrix monotone function of order $n$ preserves the Loewner order. More generally, for a fixed matrix monotone function $f$ of order $n$, one may ask for the maximal subcone of $\mathscr{P}_{2n}(\mathbb{R})$ containing $\widetilde{\mathscr{P}}_{n}(\mathbb{C})$ on which $f$ remains monotone. The analogous questions may be posed for
\begin{equation*}
\widetilde{\mathscr{P}}_{n}(\mathbb{H})
\subset
\mathscr{P}_{2n}(\mathbb{C})\,.
\end{equation*}

\end{remark}

\section{The embeddings $\Psi_{1}$ and $\Psi_{2}$ as geometric isometries}

We start this section by showing that the embeddings
\begin{equation*}
\Psi_{1}: \mathscr{P}_{n}(\mathbb{C}) \longrightarrow \mathscr{P}_{2n}(\mathbb{R})\,, \qquad \Psi_{2}: \mathscr{P}_{m}(\mathbb{H}) \longrightarrow \mathscr{P}_{2m}(\mathbb{C})
\end{equation*}
preserve the Log-Euclidean geometry.

\noindent We recall that the Frobenius norm on $\Mat_{n}(\mathbb{D})$ is defined by
\begin{equation}
\left\|\X\right\|_{\mathbb{D}} := \sqrt{\tr(\X^{*}\X)}\,, \qquad \left(\X \in \Mat_{n}(\mathbb{D})\right)\,.
\label{FrobeniusNormD}
\end{equation}

\begin{remark}

In the case $\mathbb{D} = \mathbb{H}$, we usually work with the reduced trace $\trd$, instead of the standard trace. However, for a matrix $\X \in \mathfrak{p}_{n}(\mathbb{H})$, we have 
\begin{equation*}
\trd(\X) = 2\tr(\X)\,.
\end{equation*}
Using that $\X^{*}\X \in \mathfrak{p}_{n}(\mathbb{H})$ for all $\X \in \Mat_{n}(\mathbb{H})$, it follows that
\begin{equation*}
\sqrt{\tr(\X^{*}\X)} = \frac{1}{\sqrt{2}}\sqrt{\trd(\X^{*}\X)}\,.
\end{equation*}

\end{remark}

\begin{lemma}

For all $\X \in \Mat_{n}(\mathbb{C})$ and $\Y \in \Mat_{m}(\mathbb{H})$, we have
\begin{equation*}
\left\|\Psi_{1}(\X)\right\|_{\mathbb{R}} = \sqrt{2}\left\|\X\right\|_{\mathbb{C}}\,, \qquad \left\|\Psi_{2}(\Y)\right\|_{\mathbb{C}} = \sqrt{2}\left\|\Y\right\|_{\mathbb{H}}\,.
\end{equation*}

\label{PsiOneNorm}

\end{lemma}

\begin{proof}

We prove the equality for $\Psi_{1}$, the proof for $\Psi_{2}$ is similar. Let $\X = \A+i\B$, with $\A\,, \B \in \Mat_{n}(\mathbb{R})$. Then
\begin{equation*}
\Psi_{1}(\X) = \begin{pmatrix} \A & \B \\ -\B & \A \end{pmatrix} \qquad \qquad \text{ and } \Psi_{1}(\X)^{t} = \begin{pmatrix} \A^{t} & -\B^{t} \\ \B^{t} & \A^{t} \end{pmatrix}\,,
\end{equation*}
i.e.
\begin{equation*}
\Psi_{1}(\X)^{t}\Psi_{1}(\X) = \begin{pmatrix} \A^{t}\A+\B^{t}\B & \A^{t}\B-\B^{t}\A \\ \B^{t}\A-\A^{t}\B & \B^{t}\B+\A^{t}\A \end{pmatrix}\,.
\end{equation*}
By taking the trace, we obtain 
\begin{equation*}
\left\|\Psi_{1}(\X)\right\|^{2}_{\mathbb{R}} = 2\tr(\A^{t}\A+\B^{t}\B)\,.
\end{equation*}
On the other hand,
\begin{equation*}
\X^{*}\X = \left(\A^{t}-i\B^{t}\right)\left(\A+i\B\right) = \A^{t}\A + \B^{t}\B + i\left(\A^{t}\B - \B^{t}\A\right)\,,
\end{equation*}
and using that $\A^{t}\B - \B^{t}\A$ is skew-symmetric, we get
\begin{equation*}
\tr(\X^{*}\X) = \tr(\A^{t}\A+\B^{t}\B)\,.
\end{equation*}
Thus
\begin{equation*}
\left\|\Psi_{1}(\X)\right\|^{2}_{\mathbb{R}} = 2\left\|\X\right\|^{2}_{\mathbb{C}}\,.
\end{equation*}

\end{proof}

\begin{definition}

The Log-Euclidean distance on $\mathscr{P}_{n}(\mathbb{D})$ is defined by
\begin{equation*}
\d_{\log,\mathbb{D}}(\A\,, \B) = \left\|\log_{\mathbb{D}}(\A) - \log_{\mathbb{D}}(\B)\right\|_{\mathbb{D}}\,, \qquad \left(\A\,, \B \in \mathscr{P}_{n}(\mathbb{D})\right)\,.
\end{equation*}

\end{definition}

\begin{remark}

The sets $\widetilde{\mathscr{P}}_{n}(\mathbb{C})$ and $\widetilde{\mathscr{P}}_{m}(\mathbb{H})$ are closed submanifolds of $\mathscr{P}_{2n}(\mathbb{R})$ and $\mathscr{P}_{2m}(\mathbb{C})$ respectively. We denote by $\tilde{\d}_{\log,\mathbb{R}}$ and $\tilde{\d}_{\log,\mathbb{C}}$ the maps
\begin{equation*}
\tilde{\d}_{\log,\mathbb{R}}: \widetilde{\mathscr{P}}_{n}(\mathbb{C}) \times \widetilde{\mathscr{P}}_{n}(\mathbb{C}) \mapsto \left[0, \infty\right)\,, \qquad \tilde{\d}_{\log,\mathbb{C}}: \widetilde{\mathscr{P}}_{m}(\mathbb{H}) \times \widetilde{\mathscr{P}}_{m}(\mathbb{H}) \mapsto \left[0, \infty\right)
\end{equation*}
the maps given by
\begin{equation*}
\tilde{\d}_{\log,\mathbb{R}} = \frac{1}{\sqrt{2}}(\d_{\log,\mathbb{R}})_{|_{\widetilde{\mathscr{P}}_{n}(\mathbb{C})}}\,, \qquad \tilde{\d}_{\log,\mathbb{C}} = \frac{1}{\sqrt{2}}(\d_{\log,\mathbb{C}})_{|_{\widetilde{\mathscr{P}}_{m}(\mathbb{H})}}\,.
\end{equation*}

\end{remark}

\begin{proposition}

The maps
\begin{equation*}
\Psi_{1}: \left(\mathscr{P}_{n}(\mathbb{C})\,, \d_{\log,\mathbb{C}}\right) \longrightarrow \left(\widetilde{\mathscr{P}}_{n}(\mathbb{C})\,, \tilde{\d}_{\log,\mathbb{R}}\right) 
\end{equation*}
and
\begin{equation*}
\Psi_{2}: \left(\mathscr{P}_{m}(\mathbb{H})\,, \d_{\log,\mathbb{H}}\right) \longrightarrow \left(\widetilde{\mathscr{P}}_{m}(\mathbb{H})\,, \tilde{\d}_{\log,\mathbb{C}}\right)
\end{equation*}
are isometries. More precisely, for all $\A\,,\B \in \mathscr{P}_{n}(\mathbb{C})$ and $\C\,,\D \in \mathscr{P}_{m}(\mathbb{H})$, we have
\begin{equation*}
\tilde{\d}_{\log, \mathbb{R}}\left(\Psi_{1}(\A)\,, \Psi_{1}(\B)\right) = \d_{\log,\mathbb{C}}(\A\,, \B)\,, \qquad \tilde{\d}_{\log, \mathbb{C}}\left(\Psi_{2}(\C)\,, \Psi_{2}(\D)\right) = \d_{\log,\mathbb{H}}(\C\,, \D)\,.
\end{equation*}

\end{proposition}

\begin{proof}

By the compatibility of functional calculus with $\Psi_{1}$ (see Lemma \ref{LemmaFunctionalCalculus}), that we apply to the function $f(x) = \log(x)$, we have
\begin{equation*}
\log_{\mathbb{R}}(\Psi_{1}(\A)) = \Psi_{1}(\log_{\mathbb{C}}(\A))\,,
\end{equation*}
and similarly
\begin{equation*}
\log_{\mathbb{R}}(\Psi_{1}(\B)) = \Psi_{1}(\log_{\mathbb{C}}(\B))\,.
\end{equation*}
Therefore
\begin{equation*}
\log_{\mathbb{R}}(\Psi_{1}(\A)) - \log_{\mathbb{R}}(\Psi_{1}(\B)) = \Psi_{1}\left(\log_{\mathbb{C}}(\A) - \log_{\mathbb{C}}(\B)\right)\,.
\end{equation*}
Using Lemma \ref{PsiOneNorm}, we obtain
\begin{eqnarray*}
& & \tilde{\d}_{\log,\mathbb{R}}\left(\Psi_{1}(\A),\Psi_{1}(\B)\right)^{2} = \frac{1}{2} \left\|\Psi_{1}\left(\log_{\mathbb{C}}(\A) - \log_{\mathbb{C}}(\B)\right)\right\|^{2}_{\mathbb{R}} \\
& = & \left\|\log_{\mathbb{C}}(\A)-\log_{\mathbb{C}}(\B)\right\|^{2}_{\mathbb{C}} = \d_{\log,\mathbb{C}}(\A\,, \B)^{2}\,.
\end{eqnarray*}
Taking square roots gives the result for $\Psi_{1}$. The proof for $\Psi_{2}$ is similar\,.

\end{proof}

\begin{definition}

Let $\A_{1}\,, \ldots\,, \A_{k} \in \mathscr{P}_{n}(\mathbb{D})$. Their Log-Euclidean barycentre is defined by
\begin{equation*}
\Bar_{\log,\mathbb{D}}(\A_{1}\,, \ldots\,, \A_{k}) = \exp_{\mathbb{D}}\left(\frac{1}{k}\sum\limits_{j=1}^{k}\log_{\mathbb{D}}(\A_{j})\right)\,.
\end{equation*}
More generally, if $w_{1}\,, \ldots\,, w_{k} > 0$ with $\sum\limits_{j=1}^{k}w_{j}=1$, the weighted Log-Euclidean barycentre is defined by
\begin{equation*}
\Bar^{w}_{\log,\mathbb{D}}(\A_{1}\,, \ldots\,, \A_{k}) = \exp_{\mathbb{D}}\left(\sum\limits_{j=1}^{k}w_{j}\log_{\mathbb{D}}(\A_{j})\right)\,.
\end{equation*}

\end{definition}

\begin{corollary}

Let $\A_{1}\,, \ldots\,, \A_{k} \in \mathscr{P}_{n}(\mathbb{C})\,,$ $\B_{1}\,, \ldots\,, \B_{k} \in \mathscr{P}_{m}(\mathbb{H})\,,$ and $w_{1}\,, \ldots\,, w_{k}$ positive satisfying $\sum\limits_{j=1}^{k}w_{j}=1$. Then
\begin{equation*}
\Psi_{1}\left(\Bar^{w}_{\log,\mathbb{C}}(\A_{1}\,, \ldots\,, \A_{k})\right) = \Bar^{w}_{\log,\mathbb{R}}\left(\Psi_{1}(\A_{1})\,, \ldots\,, \Psi_{1}(\A_{k})\right)
\end{equation*}
and 
\begin{equation*}
\Psi_{2}\left(\Bar^{w}_{\log,\mathbb{H}}(\B_{1}\,, \ldots\,, \B_{k})\right) = \Bar^{w}_{\log,\mathbb{C}}\left(\Psi_{2}(\B_{1})\,, \ldots\,, \Psi_{2}(\B_{k})\right)\,.
\end{equation*}

\end{corollary}

\begin{proof}

Using the compatibility of $\Psi_{1}$ with the exponential and logarithm maps, we obtain
\begin{eqnarray*}
& & \Psi_{1}\left(\Bar^{w}_{\log,\mathbb{C}}(\A_{1}\,, \ldots\,, \A_{k})\right) = \Psi_{1}\left(\exp_{\mathbb{C}}\left(\sum\limits_{j=1}^{k}w_{j}\log_{\mathbb{C}}(\A_{j})\right)\right) \\
& = & \exp_{\mathbb{R}}\left(\Psi_{1}\left(\sum\limits_{j=1}^{k}w_{j}\log_{\mathbb{C}}(\A_{j})\right)\right) = \exp_{\mathbb{R}}\left(\sum\limits_{j=1}^{k}w_{j}\Psi_{1}(\log_{\mathbb{C}}(\A_{j}))\right) \\
& = & \exp_{\mathbb{R}}\left(\sum\limits_{j=1}^{k}w_{j}\log_{\mathbb{R}}(\Psi_{1}(\A_{j}))\right) = \Bar_{\log,\mathbb{R}}^{w}\left(\Psi_{1}(\A_{1})\,, \ldots\,,\Psi_{1}(\A_{k})\right)\,.
\end{eqnarray*}
The proof for $\Psi_{2}$ is similar\,.

\end{proof}

\section{Kubo-Ando means of order $2$ and compatibility with the maps $\Psi_{i}$}

In this section, we show that every Kubo-Ando mean on $\mathscr{P}_{2}(\mathbb{D})$ admits an explicit affine expression in terms of the two matrices involved. The proof is a simple application of the functional calculus together with the Cayley-Hamilton theorem (see also \cite[Section~1]{HIGHAM}). We then show that the result naturally extends to the embedded cone $\widetilde{\mathscr{P}}_{2}(\mathbb{H}) \subset \mathscr{P}_{4}(\mathbb{C})$ and $\widetilde{\mathscr{P}}_{2}(\mathbb{C}) \subset \mathscr{P}_{4}(\mathbb{R})$\,.

\begin{lemma}

Let $\X \in \mathscr{P}_{2}(\mathbb{D})$, and let $f: \left(0, \infty\right) \to \left(0, \infty\right)$ be a continuous function. Then there exists scalars $\alpha_{f}(\X), \beta_{f}(\X)\in \mathbb{R}$ such that
\begin{equation*}
f(\X) = \alpha_{f}(\X)\Id + \beta_{f}(\X)\X\,,
\end{equation*}
with $\alpha_{f}(\X)\,, \beta_{f}(\X)$ unique if $\left|\Spec(\X)\right| = 2$.
\label{LemmaHigham}

\end{lemma}

\begin{proof}

Let $\X \in \mathscr{P}_{2}(\mathbb{D})$. There exists $\U \in \U_{\mathbb{D}}$ and $\Lambda = \diag(\lambda_{+}\,, \lambda_{-})$ such that $\X = \U\Lambda\U^{*}$, and where $\lambda_{+}\,, \lambda_{-} > 0$ are the eigenvalues of $\X$. By the functional calculus, we get
\begin{equation*}
f(\X) = \U \begin{pmatrix} f(\lambda_{+}) & 0 \\ 0 & f(\lambda_{-}) \end{pmatrix} \U^{*}\,.
\end{equation*}
Assume first that $\lambda_{+} \neq \lambda_{-}$. Let
\begin{equation*}
r(t) = \alpha_{f}(\X) + \beta_{f}(\X) t\,, \qquad \left(t \in \mathbb{R}\right)\,,
\end{equation*}
be the unique affine polynomial satisfying
\begin{equation*}
r(\lambda_{+}) = f(\lambda_{+})\,, \qquad r(\lambda_{-}) = f(\lambda_{-})\,.
\end{equation*}
In particular, we get
\begin{equation*}
\begin{cases} \alpha_{f}(\X) + \beta_{f}(\X) \lambda_{+} & = f(\lambda_{+}) \\ \alpha_{f}(\X) + \beta_{f}(\X) \lambda_{-} & = f(\lambda_{-}) \end{cases}
\end{equation*}
i.e.
\begin{equation*}
\beta_{f}(\X) = \frac{f(\lambda_{+}) - f(\lambda_{-})}{\lambda_{+}-\lambda_{-}}\,, \qquad \alpha_{f}(\X) = \frac{\lambda_{+}f(\lambda_{-}) - \lambda_{-}f(\lambda_{+})}{\lambda_{+}-\lambda_{-}}\,.
\end{equation*}
Since
\begin{equation*}
r(\X) = \U \begin{pmatrix} r(\lambda_{+}) & 0 \\ 0 & r(\lambda_{-}) \end{pmatrix} \U^{*}\,,
\end{equation*}
we obtain
\begin{equation*}
r(\X) = \U \begin{pmatrix} f(\lambda_{+}) & 0 \\ 0 & f(\lambda_{-}) \end{pmatrix} \U^{*} = f(\X)\,,
\end{equation*}
i.e.
\begin{equation*}
f(\X) = \alpha_{f}(\X) \Id_{2} + \beta_{f}(\X) \X\,.
\end{equation*}
Now assume that $\lambda_{+} = \lambda_{-}$. Since $\X$ is Hermitian, we get that $\X = \lambda_{+} \Id_{2}$. Hence, we choose $\alpha_{f}(\X) = f(\lambda_{+})$ and $\beta_{f}(\X) = 0$, and we get $f(\X) = \alpha_{f}(\X) \Id_{2} + \beta_{f}(\X) \X$\,.

\end{proof}

\begin{remark}

The preceding lemma may also be viewed as a consequence of the Cayley-Hamilton theorem. Indeed, every matrix $\X \in \Mat_{2}(\mathbb{C})$ satisfies
\begin{equation*}
\X^{2} - \tr(\X)\X + \det(\X)\Id_{2} = 0\,.
\end{equation*}
Thus the algebra generated by $\X$ is at most two-dimensional (i.e. $\Span\left\{\Id_{2}\,, \X\right\}$). Therefore, every matrix function of $\X$ belongs to the same two-dimensional algebra\,.

\end{remark}

\noindent We now apply Lemma \ref{LemmaHigham} to Kubo-Ando means of order $2$\,.

\begin{theo}

Let $\sigma$ be a Kubo-Ando mean of order $2$ with representing function $f: \left(0, \infty\right) \mapsto \left(0, \infty\right)$. For all $\A\,, \B \in \mathscr{P}_{2}(\mathbb{D})$, let $\X := \A^{-\frac{1}{2}}\B\A^{-\frac{1}{2}}$. Then there exist scalars $\alpha_{f}(\X)\,, \beta_{f}(\X) \in \mathbb{R}$ such that
\begin{equation*}
\A \sigma \B = \alpha_{f}(\X)\A + \beta_{f}(\X)\B\,,
\end{equation*}
with $\alpha_{f}(\X)\,, \beta_{f}(\X)$ unique if $\left|\Spec(\X)\right| = 2$. More precisely, if $\Spec(\X) = \left\{\lambda_{+}\,, \lambda_{-}\right\}$, with $\lambda_{+} \neq \lambda_{-}$, we have
\begin{equation*}
\alpha_{f}(\X) = \frac{\lambda_{+}f(\lambda_{-}) - \lambda_{-}f(\lambda_{+})}{\lambda_{+} - \lambda_{-}} \qquad \beta_{f}(\X) = \frac{f(\lambda_{+})-f(\lambda_{-})}{\lambda_{+}-\lambda_{-}}
\end{equation*}
and in the case $\lambda_{+} = \lambda_{-}$, we take $\alpha_{f}(\X) = f(\lambda_{+})$ and $\beta_{f}(\X) = 0$.

\label{2x2KuboAndo}

\end{theo}

\begin{proof}

Using Theorem \ref{TheoremRepresentation}, we have
\begin{equation*}
\A \sigma \B = \A^{\frac{1}{2}} f(\X) \A^{\frac{1}{2}}\,.
\end{equation*}
It follows from Lemma \ref{LemmaHigham} that
\begin{equation*}
f(\X) = \alpha_{f}(\X)\Id_{2} + \beta_{f}(\X)\X\,.
\end{equation*}
Therefore
\begin{equation*}
\A \sigma \B = \A^{\frac{1}{2}}\left(\alpha_{f}(\X)\Id_{2} + \beta_{f}(\X)\X\right)\A^{\frac{1}{2}}\,.
\end{equation*}
Using $\A^{\frac{1}{2}}\Id_{2}\A^{\frac{1}{2}} = \A$ and $\A^{\frac{1}{2}}\X\A^{\frac{1}{2}} = \B$, we obtain
\begin{equation*}
\A \sigma \B = \alpha_{f}(\X)\A+\beta_{f}(\X)\B\,.
\end{equation*}

\end{proof}

\begin{remark}

The coefficients appearing in Theorem~\ref{2x2KuboAndo} depend only on the eigenvalues of $\X$ and the function $f$ itself. Note that the eigenvalues of $\X$ only depends on $\tr(\X)$ and $\det(\X)$. Indeed, the characteristic polynomial $\P_{\X}$ of $\X$ is given by
\begin{equation*}
\P_{\X}(\lambda) = \lambda^{2} - \tr(\X)\lambda + \det(\X)\,, \qquad \left(\lambda \in \mathbb{R}\right)\,.
\end{equation*}
Therefore, we get
\begin{equation*}
\lambda_{\pm} = \frac{\tr(\X) \pm \sqrt{\tr(\X)^{2}-4\det(\X)}}{2}\,.
\end{equation*}
Note that if $\mathbb{D} = \mathbb{H}$, $\det(\X)$ corresponds to the Moore determinant of $\X$ (and $\tr(\X) = \frac{1}{2}\trd(\X)$)\,.

\end{remark}

\begin{example}

We denote by $\sharp$ the geometric mean on $\mathscr{P}_{2}(\mathbb{C})$ given by
\begin{equation*}
\A \sharp \B = \A^{\frac{1}{2}}\left(\A^{-\frac{1}{2}}\B\A^{-\frac{1}{2}}\right)^{\frac{1}{2}}\A^{\frac{1}{2}}\,,
\end{equation*}
i.e. the representing function is $f(x) = \sqrt{x}$. We assume that $\Spec(\X) = \left\{\lambda_{+}\,, \lambda_{-}\right\}$, with $\lambda_{+} \neq \lambda_{-}$. Therefore, it follows from Theorem~\ref{2x2KuboAndo} that
\begin{equation*}
\A \sharp \B = \alpha_{f}(\X)\A+\beta_{f}(\X)\B,
\end{equation*}
where
\begin{equation*}
\beta_{f}(\X) = \frac{\sqrt{\lambda_{+}}-\sqrt{\lambda_{-}}}{\lambda_{+}-\lambda_{-}}\,, \quad \text{ and } \quad \alpha_{f}(\X) = \frac{\lambda_{+}\sqrt{\lambda_{-}} - \lambda_{-}\sqrt{\lambda_{+}}}{\lambda_{+}-\lambda_{-}}\,.
\end{equation*}
We have
\begin{equation*}
\lambda_{+}-\lambda_{-} = \left(\sqrt{\lambda_{+}}-\sqrt{\lambda_{-}}\right)\left(\sqrt{\lambda_{+}}+\sqrt{\lambda_{-}}\right)\,,
\end{equation*}
so
\begin{equation*}
\beta_{f}(\X) = \frac{1}{\sqrt{\lambda_{+}}+\sqrt{\lambda_{-}}}\,.
\end{equation*}
Similarly,
\begin{equation*}
\alpha_{f}(\X) = \frac{\lambda_{+}\sqrt{\lambda_{-}}-\lambda_{-}\sqrt{\lambda_{+}}}{\lambda_{+}-\lambda_{-}} = \frac{\sqrt{\lambda_{+}\lambda_{-}}\left(\sqrt{\lambda_{+}}-\sqrt{\lambda_{-}}\right)}{\left(\sqrt{\lambda_{+}}-\sqrt{\lambda_{-}}\right)\left(\sqrt{\lambda_{+}}+\sqrt{\lambda_{-}}\right)} = \frac{\sqrt{\lambda_{+}\lambda_{-}}}{\sqrt{\lambda_{+}}+\sqrt{\lambda_{-}}}\,.
\end{equation*}
Therefore
\begin{equation*}
\A\sharp\B = \frac{\B+\sqrt{\lambda_{+}\lambda_{-}}\A}{\sqrt{\lambda_{+}}+\sqrt{\lambda_{-}}}\,.
\end{equation*}
Using that
\begin{equation*}
\lambda_{+}\lambda_{-} = \det(\X)\,, \qquad \lambda_{+} + \lambda_{-} = \tr(\X)\,,
\end{equation*}
i.e.
\begin{equation*}
\left(\sqrt{\lambda_{+}}+\sqrt{\lambda_{-}}\right)^{2} = \tr(\X)+2\sqrt{\det(\X)}\,.
\end{equation*}
Finally
\begin{equation*}
\A\sharp\B = \frac{\B+\sqrt{\det(\X)}\A}{\sqrt{\tr(\X)+2\sqrt{\det(\X)}}} \qquad \left(\X = \A^{-1/2}\B\A^{-1/2}\right)\,.
\end{equation*}
This is precisely the classical Pusz-Woronowicz formula for the geometric mean of two positive definite $2\times 2$ matrices (see \cite{PUSZWORONOWICZ}). An analogous result has been obtained in \cite{CHOIKIMLIM}\,.

\end{example}

\noindent The previous theorem extends immediately to the embedded cones
\begin{equation*}
\widetilde{\mathscr{P}}_{2}(\mathbb{H}) \subset \mathscr{P}_{4}(\mathbb{C})\,, \qquad \widetilde{\mathscr{P}}_{2}(\mathbb{C}) \subset \mathscr{P}_{4}(\mathbb{R})\,.
\end{equation*}

\begin{proposition}

Let $\sigma$ be a Kubo-Ando mean of order $2$ on $\mathscr{P}_{2}(\mathbb{C})$ with representing function $f$, and let $\widetilde{\sigma}$ be the corresponding mean on $\widetilde{\mathscr{P}}_{2}(\mathbb{C})$. Let $\X\,, \Y \in \widetilde{\mathscr{P}}_{2}(\mathbb{C})$ and let
\begin{equation*}
\T := \X^{-\frac{1}{2}}\Y\X^{-\frac{1}{2}}\,.
\end{equation*}
Let $\lambda_{+}$ and $\lambda_{-}$ be the eigenvalues of $\T$. If $\lambda_{+} \neq \lambda_{-}$, 
\begin{equation*}
\X \widetilde{\sigma} \Y = \alpha_{f}(\T)\X + \beta_{f}(\T)\Y\,,
\end{equation*}
where
\begin{equation*}
\alpha_{f}(\T) = \frac{\lambda_{+}f(\lambda_{-})-\lambda_{-}f(\lambda_{+})}{\lambda_{+}-\lambda_{-}}\,, \qquad \qquad \beta_{f}(\T) = \frac{f(\lambda_{+})-f(\lambda_{-})}{\lambda_{+}-\lambda_{-}}\,.
\end{equation*}
Moreover
\begin{equation*}
\lambda_{\pm} := \frac{\frac{1}{2}\tr(\T)\pm \sqrt{\left(\frac{1}{2}\tr(\T)\right)^2 -4\det(\T)^{\frac{1}{2}}}}{2}\,.
\end{equation*}
If $\lambda_{+} = \lambda_{-}$, then 
\begin{equation*}
\X \widetilde{\sigma} \Y = f(\lambda_{+})\X\,.
\end{equation*}

\label{DecompositionMeanAlphaBeta}

\end{proposition}

\begin{proof}

Since $\X\,, \Y \in \widetilde{\mathscr{P}}_{2}(\mathbb{C})$, there exist $\A\,, \B \in \mathscr{P}_{2}(\mathbb{C})$ such that
\begin{equation*}
\X = \Psi_{1}(\A)\,, \qquad \Y = \Psi_{1}(\B)\,.
\end{equation*}
Therefore, it follows that
\begin{equation*}
\T = \X^{-\frac{1}{2}}\Y\X^{-\frac{1}{2}} = \Psi_{1}\left(\A^{-\frac{1}{2}}\B\A^{-\frac{1}{2}}\right)\,.
\end{equation*}
Let $\Z$ be the element of $\mathscr{P}_{2}(\mathbb{C})$ given by
\begin{equation*}
\Z := \A^{-\frac{1}{2}}\B\A^{-\frac{1}{2}}\,,
\end{equation*}
i.e. $\T = \Psi_{1}(\Z)$. By definition of the corresponding mean $\widetilde{\sigma}$, we have
\begin{equation*}
\X \widetilde{\sigma} \Y = \X^{\frac{1}{2}} f(\T) \X^{\frac{1}{2}}\,.
\end{equation*}
It follows from Lemma \ref{LemmaFunctionalCalculus} that
\begin{equation*}
f(\T) = f(\Psi_{1}(\Z)) = \Psi_{1}(f(\Z))\,.
\end{equation*}
We have $\Spec(\Z) = \left\{\lambda_{+}\,, \lambda_{-}\right\}$. Therefore, Theorem \ref{2x2KuboAndo} implies that
\begin{equation*}
f(\Z)=\alpha_{f}(\Z)\Id_{2} + \beta_{f}(\Z)\Z\,,
\end{equation*}
where $\alpha_{f}(\Z)$ and $\beta_{f}(\Z)$ are given by
\begin{equation*}
\alpha_{f}(\Z) = \frac{\lambda_{+}f(\lambda_{-}) - \lambda_{-}f(\lambda_{+})}{\lambda_{+} - \lambda_{-}}\,, \qquad \beta_{f}(\Z) = \frac{f(\lambda_{+})-f(\lambda_{-})}{\lambda_{+}-\lambda_{-}}\,.
\end{equation*}
Then
\begin{equation*}
f(\T) = f(\Psi_{1}(\Z)) = \Psi_{1}(f(\Z)) = \alpha_{f}(\Z)\Psi_{1}(\Id_{2}) + \beta_{f}(\Z)\Psi_{1}(\Z) = \alpha_{f}(\Z)\Id_{4} + \beta_{f}(\Z)\T\,.
\end{equation*}
i.e.
\begin{eqnarray*}
\X \widetilde{\sigma} \Y & = & \X^{\frac{1}{2}}f(\T)\X^{\frac{1}{2}} = \X^{\frac{1}{2}}\left(\alpha_{f}(\Z)\Id_{4} + \beta_{f}(\Z)\T\right)\X^{\frac{1}{2}} \\
& = & \alpha_{f}(\Z)\X +\beta_{f}(\Z)\X^{\frac{1}{2}}\T\X^{\frac{1}{2}} = \alpha_{f}(\Z)\X\ + \beta_{f}(\Z)\Y\,.
\end{eqnarray*}
It remains to express $\lambda_{\pm}$ in terms of $\tr(\T)$ and $\det(\T)$. Since $\T = \Psi_{1}(\Z)$, each eigenvalue of $\Z$ appears twice as an eigenvalue of $\T$. Hence
\begin{equation*}
\tr(\T) = 2\lambda_{+} + 2\lambda_{-} = 2\left(\lambda_{+} + \lambda_{-}\right)\,, \qquad \det(\T) = \lambda^{2}_{+}\lambda^{2}_{-}\,.
\end{equation*}
Thus
\begin{equation*}
\lambda_{+} + \lambda_{-} = \frac{1}{2}\tr(\T)\,, \qquad \lambda_{+}\lambda_{-} = \det(\T)^{\frac{1}{2}}\,.
\end{equation*}
i.e. $\lambda_{+}$ and $\lambda_{-}$ are the roots of
\begin{equation*}
t^{2} - \frac{1}{2}\tr(\T)t + \det(\T)^{\frac{1}{2}} = 0\,.
\end{equation*}
Finally, using the quadratic formula, we obtain
\begin{equation*}
\lambda_{\pm} = \frac{\frac{1}{2}\tr(\T) \pm \sqrt{\left(\frac{1}{2}\tr(\T)\right)^{2} -4\det(\T)^{\frac{1}{2}}}}{2}\,.
\end{equation*}
In the case $\lambda_{+}=\lambda_{-}=\lambda$, we have
\begin{equation*}
\T=\lambda\Id_{4}\,.
\end{equation*}
Indeed, using the notation of the proof of Proposition~\ref{DecompositionMeanAlphaBeta}, we have $\T=\Psi_{1}(\Z)$, where $\Z\in\mathscr{P}_{2}(\mathbb{C})$ has a unique eigenvalue $\lambda$. Since $\Z$ is Hermitian, it follows that
\begin{equation*}
\Z=\lambda\Id_{2}\,,
\end{equation*}
and therefore
\begin{equation*}
\T=\Psi_{1}(\Z)=\lambda\Id_{4}\,.
\end{equation*}
Finally
\begin{equation*}
f(\T)=f(\lambda)\Id_{4}\,,
\end{equation*}
and hence
\begin{equation*}
\X\widetilde{\sigma}\Y = \X^{\frac{1}{2}}f(\T)\X^{\frac{1}{2}} = f(\lambda)\X\,.
\end{equation*}

\end{proof}

\begin{remark}

The previous proposition does not extend to arbitrary elements of $\mathscr{P}_{4}(\mathbb{R})$. Indeed, there exist $\X\,, \Y \in \mathscr{P}_{4}(\mathbb{R})$ for which $\X\sharp\Y$ cannot be written as a linear combination of $\X$ and $\Y$. Consider
\begin{equation*}
\X=\Id_{4} \qquad \text{ and } \qquad \Y=\diag(1\,,2\,,3\,,4)\,.
\end{equation*}
Since $\X$ and $\Y$ commute, we have
\begin{equation*}
\X\sharp\Y=\Y^{\frac{1}{2}} = \diag(1\,,\sqrt{2}\,,\sqrt{3}\,,2)\,.
\end{equation*}
Suppose that there exist $\alpha\,,\beta\in\mathbb{R}$ such that
\begin{equation*}
\X\sharp\Y=\alpha\X+\beta\Y\,.
\end{equation*}
Then
\begin{equation*}
\diag(1\,,\sqrt{2}\,,\sqrt{3}\,,2) = \diag(\alpha+\beta\,,\alpha+2\beta\,,\alpha+3\beta\,,\alpha+4\beta)\,.
\end{equation*}
Comparing the first and last diagonal entries gives
\begin{equation*}
1=\alpha+\beta\,, \qquad 2=\alpha+4\beta\,,
\end{equation*}
and therefore
\begin{equation*}
\alpha=\frac{2}{3}\,, \qquad \beta=\frac{1}{3}\,.
\end{equation*}
However, comparing the second diagonal entries would then give
\begin{equation*}
\sqrt{2} = \alpha+2\beta = \frac{4}{3}\,,
\end{equation*}
which is a contradiction. Thus $\X \sharp \Y$ does not belong to $\Span\{\X\,,\Y\}$\,.

\end{remark}

\section{Further remarks and application results}

We recall the celebrated Fan-Hoffman theorem, which states that for a matrix $\A \in \Mat_{n}(\mathbb{C})$, there exists a unique Hermitian matrix minimizing the distance
\begin{equation*}
\left\|\A-\X\right\|_{\mathbb{C}}\,, \qquad \X=\X^{*}\,,
\end{equation*}
with respect to the Frobenius norm defined in Equation~\eqref{FrobeniusNormD}. This minimizer is the Hermitian part of $\A$, given by
\begin{equation*}
\frac{1}{2}\left(\A+\A^{*}\right)
\end{equation*}
(see \cite[Theorem~8.4]{HIGHAM}). In view of the preceding results, it is natural to ask whether an arbitrary element of $\mathscr{P}_{2n}(\mathbb{R})$ admits a unique closest element in the embedded cone $\widetilde{\mathscr{P}}_{n}(\mathbb{C})$\,.

\medskip

\noindent Recall that
\begin{equation*}
\K_{2n} = \begin{bmatrix} 0 & \Id_{n} \\ -\Id_{n} & 0 \end{bmatrix}\,,
\end{equation*}
and that the image of the embedding $\Psi_{1}$ consists precisely of the matrices in $\mathscr{P}_{2n}(\mathbb{R})$ that commute with $\K_{2n}$. Consider an arbitrary matrix $\A\in\mathscr{P}_{2n}(\mathbb{R})$, written in block form as
\begin{equation*}
\A = \begin{bmatrix} \X & \Y \\ \Y^{t} & \W \end{bmatrix}\,, \qquad \X=\X^{t}\,, \qquad \W=\W^{t}\,.
\end{equation*}
We define
\begin{equation*}
\P(\A) = \frac{1}{2}\left(\A+\K_{2n}^{t}\A\K_{2n}\right)\,.
\end{equation*}
A direct computation gives
\begin{equation*}
\K_{2n}^{t}\A\K_{2n} = \begin{bmatrix} \W & -\Y^{t} \\ -\Y & \X \end{bmatrix}\,,
\end{equation*}
and therefore
\begin{equation*}
\P(\A) = \begin{bmatrix} \frac{1}{2}(\X+\W) & \frac{1}{2}(\Y-\Y^{t}) \\ \frac{1}{2}(\Y^{t}-\Y) & \frac{1}{2}(\X+\W)
\end{bmatrix}\,.
\end{equation*}
In particular, $\P(\A)$ commutes with $\K_{2n}$ and has the form
\begin{equation*}
\P(\A) = \begin{bmatrix} \widehat{\X} & \widehat{\Y} \\ -\widehat{\Y} & \widehat{\X} \end{bmatrix}\,, \qquad \widehat{\X}^{t} = \widehat{\X}\,, \qquad \widehat{\Y}^{t}=-\widehat{\Y}\,.
\end{equation*}
Moreover, since $\A$ is positive definite and $\K_{2n}$ is orthogonal, the matrix $\K_{2n}^{t}\A\K_{2n}$ is also positive definite. Hence
\begin{equation*}
\P(\A) = \frac{1}{2}\left(\A+\K_{2n}^{t}\A\K_{2n}\right) > 0\,,
\end{equation*}
and therefore
\begin{equation*}
\P(\A)\in\widetilde{\mathscr{P}}_{n}(\mathbb{C})\,.
\end{equation*}
We now show that $\P(\A)$ is precisely the unique closest element of $\widetilde{\mathscr{P}}_{n}(\mathbb{C})$ to $\A$ with respect to the Frobenius norm.

\begin{proposition}

Let $\A \in \mathscr{P}_{2n}(\mathbb{R})$ be written in block form as
\begin{equation*}
\A = \begin{bmatrix} \X & \Y\\ \Y^{t} & \W \end{bmatrix}\,, \qquad \X=\X^{t}\,,\qquad \W=\W^{t}\,.
\end{equation*}
Then the unique closest element of $\widetilde{\mathscr{P}}_{n}(\mathbb{C})$ to $\A$, with respect to the Frobenius norm, is
\begin{equation*}
\P(\A) = \frac{1}{2}\left(\A+\K_{2n}^{t}\A\K_{2n}\right)\,.
\end{equation*}
Similarly,
\begin{equation*}
\P(\A) = \begin{bmatrix} \frac{1}{2}(\X+\W) & \frac{1}{2}(\Y-\Y^{t}) \\
\frac{1}{2}(\Y^{t}-\Y) & \frac{1}{2}(\X+\W) \end{bmatrix}\,.
\end{equation*}

\end{proposition}

\begin{proof}

Let
\begin{equation*}
\mathscr{S} = \left\{\M\in\Sym_{2n}(\mathbb{R})\,;\, \M\K_{2n}=\K_{2n}\M\right\}\,.
\end{equation*}
Thus
\begin{equation*}
\widetilde{\mathscr{P}}_{n}(\mathbb{C}) = \mathscr{S} \cap \mathscr{P}_{2n}(\mathbb{R})\,.
\end{equation*}
Consider the linear map
\begin{equation*}
\tau:\Sym_{2n}(\mathbb{R}) \longrightarrow \Sym_{2n}(\mathbb{R})\,,
\qquad \tau(\M)=\K_{2n}^{t}\M\K_{2n}\,.
\end{equation*}
Since $\K_{2n}$ is orthogonal, $\tau$ is an isometry with respect to the Frobenius inner product. Moreover, $\tau^{2}=\Id$, and its fixed-point space is precisely $\mathscr{S}$. It follows that
\begin{equation*}
\P = \frac{1}{2}\left(\Id + \tau\right)
\end{equation*}
is the orthogonal projection of $\Sym_{2n}(\mathbb{R})$ onto $\mathscr{S}$ with respect to the Frobenius inner product.

\noindent Therefore, for every $\M \in \mathscr{S}$, we have
\begin{equation*}
\left\|\A-\M\right\|_{\mathbb{R}}^{2} = \left\|\A - \P(\A)\right\|_{\mathbb{R}}^{2} + \left\|\P(\A)-\M\right\|_{\mathbb{R}}^{2}\,.
\end{equation*}
Therefore,
\begin{equation*}
\left\|\A-\M\right\|_{\mathbb{R}} \geq \left\|\A-\P(\A)\right\|_{\mathbb{R}}\,,
\end{equation*}
with equality if and only if $\M=\P(\A)$.

\noindent As observed above, $\P(\A)$ is positive definite, and hence
\begin{equation*}
\P(\A)\in\widetilde{\mathscr{P}}_{n}(\mathbb{C})\,.
\end{equation*}
Since $\widetilde{\mathscr{P}}_{n}(\mathbb{C})\subset\mathscr{S}$, it follows that $\P(\A)$ is also the unique closest element of $\widetilde{\mathscr{P}}_{n}(\mathbb{C})$ to $\A$\,.

\end{proof}

\begin{remark}

The preceding proposition may be viewed as an analogue of the Fan-Hoffman theorem for the embedding $\Psi_{1}$. In both cases, the nearest-point map is obtained by averaging with an involution. For the space of Hermitian matrices, one considers
\begin{equation*}
\A\longmapsto\frac{1}{2}\left(\A+\A^{*}\right)\,,
\end{equation*}
whereas for the embedded cone $\widetilde{\mathscr{P}}_{n}(\mathbb{C})$, the corresponding projection is
\begin{equation*}
\A \longmapsto \frac{1}{2}\left(\A+\K_{2n}^{t}\A\K_{2n}\right)\,.
\end{equation*}

\end{remark}

\appendix

\section{Matrix monotonicity of order $n$ over $\mathbb{H}$}
 
\label{AppendixA}
 
This appendix completes the proof of Lemma \ref{LemmaEquivalenceRCH} by proving the equivalence between the complex and quaternionic cases. If $n = 1$, there is nothing to prove: a $1 \times 1$ quaternionic Hermitian matrix is real, so $\mathscr{P}_{1}(\mathbb{D}) = \left(0, \infty\right)$ for every $\mathbb{D}$, and each of the three conditions says that $f$ is nondecreasing. We assume from now on that $n \geq 2$.
 
\begin{lemma}
 
We have $\mathscr{P}_{n}(\mathbb{C}) \subseteq \mathscr{P}_{n}(\mathbb{H})$, and the Loewner orders agree on $\mathscr{P}_{n}(\mathbb{C})$. Therefore, if $f$ is matrix monotone of order $n$ on $\mathscr{P}_{n}(\mathbb{H})$, then $f$ is matrix monotone of order $n$ on $\mathscr{P}_{n}(\mathbb{C})$.
 
\label{LemmaInclusionCH}
 
\end{lemma}
 
\begin{proof}
 
Let $\H \in \mathfrak{p}_{n}(\mathbb{C}) \subseteq \mathfrak{p}_{n}(\mathbb{H})$. Then $\Psi_{2}(\H) = \diag(\H\,, \overline{\H})$, and $\overline{\H}$ has the same (real) eigenvalues as $\H$. Hence $\Psi_{2}(\H) \geq 0$ (resp. $> 0$) if and only if $\H \geq 0$ (resp. $> 0$) over $\mathbb{C}$. By Lemma~\ref{LemmaLoewnerEmbedding} (resp. Proposition~\ref{PropositionHtoC}), $\Psi_{2}(\H) \geq 0$ (resp. $> 0$) if and only if $\H \geq 0$ (resp. $> 0$) over $\mathbb{H}$. Applying this to $\H \in \mathscr{P}_{n}(\mathbb{C})$ and to $\H = \B - \A$ gives both claims; the last statement follows by restriction.
 
\end{proof}
 
\begin{proof}[Proof of Lemma~\ref{LemmaEquivalenceRCH}, case $\mathbb{D}=\mathbb{H}$]
 
If $f$ is matrix monotone of order $n$ on $\mathscr{P}_{n}(\mathbb{H})$, it is matrix monotone of order $n$ on $\mathscr{P}_{n}(\mathbb{C})$ by Lemma~\ref{LemmaInclusionCH}.
 
\noindent Conversely, assume that $f$ is matrix monotone of order $n$ on $\mathscr{P}_{n}(\mathbb{C})$. Since $n \geq 2$, $f$ is matrix monotone of order $2$, hence continuously differentiable on $\left(0, \infty\right)$ (see \cite{HIAI2010}). By Remark \ref{RemarkLoewner}, $\L_{f}(x_{1}\,, \ldots\,, x_{n}) \geq 0$ for all $x_{1}\,, \ldots\,, x_{n} > 0$.
 
\noindent Fix $\A \leq \B$ in $\mathscr{P}_{n}(\mathbb{H})$, set $\G = \B - \A \geq 0$ and $\A_{t} = \A + t\G \in \mathscr{P}_{n}(\mathbb{H})$ for $t \in \left[0\,, 1\right]$, and let
\begin{equation*}
\X_{t} := \Psi_{2}(\A_{t}) = \Psi_{2}(\A) + t\Psi_{2}(\G) \in \mathscr{P}_{2n}(\mathbb{C})\,.
\end{equation*}
By Lemma \ref{LemmaLoewnerEmbedding}, $\Psi_{2}(\G) \geq 0$. For each $t$, write $\A_{t} = \U_{t}\Lambda_{t}\U_{t}^{*}$ with $\U_{t} \in \Sp(n)$ and $\Lambda_{t} = \diag(\lambda(t))$, $\lambda(t) \in \left(0, \infty\right)^{n}$. As in the proof of Lemma~\ref{LemmaLoewnerEmbedding}, $\V_{t} := \Psi_{2}(\U_{t})$ is unitary and
\begin{equation*}
\X_{t} = \V_{t}\diag(\Lambda_{t}\,, \Lambda_{t})\V_{t}^{*}\,.
\end{equation*}
By the Daleckii-Krein formula over $\mathbb{C}$ (see \cite[Theorem~V.3.3]{BHATIA2}), which holds at every point for $f \in \mathscr{C}^{1}$, the map $t \mapsto f(\X_{t})$ is differentiable on $\left[0\,, 1\right]$ with
\begin{equation*}
\frac{d}{dt}f(\X_{t}) = \V_{t}\Big(\L_{f}\big(\lambda(t)\,, \lambda(t)\big) \circ \big(\V_{t}^{*}\Psi_{2}(\G)\V_{t}\big)\Big)\V_{t}^{*}\,,
\end{equation*}
where $\L_{f}(\lambda(t)\,, \lambda(t))$ is the $2n \times 2n$ Loewner matrix, with $\lambda(t) = \left(\lambda_{1}(t)\,, \ldots\,, \lambda_{n}(t)\right)$. Writing $\L_{t} := \L_{f}(\lambda(t)) \geq 0$, we have
\begin{equation*}
\L_{f}\big(\lambda(t)\,, \lambda(t)\big) = \begin{pmatrix} \L_{t} & \L_{t} \\ \L_{t} & \L_{t} \end{pmatrix} = \begin{pmatrix} \Id_{n} \\ \Id_{n} \end{pmatrix} \L_{t} \begin{pmatrix} \Id_{n} & \Id_{n} \end{pmatrix} \geq 0\,.
\end{equation*}
Since $\V_{t}^{*}\Psi_{2}(\G)\V_{t} \geq 0$, the Schur product theorem gives $\frac{d}{dt}f(\X_{t}) \geq 0$. Hence, for every $x \in \mathbb{C}^{2n}$, the real function $\varphi(t) = x^{*}f(\X_{t})x$ is differentiable on $\left[0\,, 1\right]$ with $\varphi' \geq 0$, so $\varphi(1) \geq \varphi(0)$ by the mean value theorem. Therefore $f(\Psi_{2}(\A)) = f(\X_{0}) \leq f(\X_{1}) = f(\Psi_{2}(\B))$. By Lemma \ref{LemmaFunctionalCalculus},
\begin{equation*}
\Psi_{2}(f(\A)) \leq \Psi_{2}(f(\B))\,,
\end{equation*}
and Lemma~\ref{LemmaLoewnerEmbedding} gives $f(\A) \leq f(\B)$.
 
\end{proof}
 
\begin{remark}
 
The proof only uses the complex Daleckii-Krein formula, applied to $\Psi_{2}(\A_{t})$. The key point is that the eigenvalues of $\Psi_{2}(\A_{t})$ are precisely those of $\A_{t}$, each repeated twice. Therefore, the relevant $2n \times 2n$ Loewner matrix is obtained by duplicating the $n\times n$ Loewner matrix.
 
\end{remark}

\end{document}